\documentclass[twoside]{article}

\usepackage[accepted]{aistats2020}

\usepackage[T1]{fontenc}
\usepackage{aecompl}

\usepackage{xr}
\makeatletter
\newcommand*{\addFileDependency}[1]{%
  \typeout{(#1)}
  \@addtofilelist{#1}
  \IfFileExists{#1}{}{\typeout{No file #1.}}
}
\makeatother
 
\newcommand*{\myexternaldocument}[1]{%
    \externaldocument{#1}%
    \addFileDependency{#1.tex}%
    \addFileDependency{#1.aux}%
}

\myexternaldocument{appendix}

\usepackage[square,sort,comma,numbers]{natbib}
\usepackage[utf8]{inputenc} %
\usepackage[T1]{fontenc}    %
\usepackage{hyperref}       %
\usepackage{url}            %
\usepackage{booktabs}       %
\usepackage{amsfonts}       %
\usepackage{nicefrac}       %
\usepackage{microtype}      %

\usepackage{setspace}
\usepackage{microtype}
\usepackage{graphicx}
\usepackage{subfigure}
\usepackage{booktabs}
\usepackage{amssymb,amsfonts}
\usepackage{amsmath, amsthm}
\usepackage[algo2e, ruled]{algorithm2e}
\usepackage{hyperref}
\usepackage{amsfonts}
\usepackage{multirow}
\usepackage{makecell}
\usepackage{algorithm} 
\usepackage{algorithmic}  
\usepackage{amsmath}
\usepackage{mathtools}
\usepackage{amsfonts}
\usepackage{microtype}
\usepackage{graphicx}
\usepackage{booktabs} %
\usepackage{caption}
\usepackage{comment}

\usepackage{pdfpages}
\usepackage{thmtools}
\usepackage{thm-restate}
\usepackage{siunitx}
\usepackage{dsfont}
\usepackage{enumerate}   
\usepackage{paralist}
\usepackage{mathtools}

\usepackage{cleveref}

\newcommand{\norm}[1]{\left\lVert#1\right\rVert}
\newcommand{\ip}[2]{\langle #1,#2 \rangle}

\newcommand{\pde}[1]{\frac{\partial}{\partial #1}}
\usepackage{stackengine}

\makeatletter
\newcommand{\distas}[1]{\mathbin{\overset{#1}{\kern\z@\sim}}}%

\newtheorem{theorem}{Theorem}

\newtheorem{lemma}{Lemma}
\newtheorem{remark}{Remark}
\newtheorem{cor}{Corollary}
\newtheorem{proposition}{Proposition}

\newcommand{\bas}[1]{\begin{align*}#1\end{align*}}
\newcommand{\ba}[1]{\begin{align}#1\end{align}}

\newcommand{\cL}{\mathcal{L}}

\newcommand{\kl}{\text{KL}}

\newcommand{\beqs}{\vspace{0mm}\begin{eqnarray}}
\newcommand{\eeqs}{\vspace{0mm}\end{eqnarray}}
\newcommand{\barr}{\begin{array}}
\newcommand{\earr}{\end{array}}

\newcommand{\E}{\mathbb{E}}

\definecolor{alizarin}{rgb}{0.82, 0.1, 0.26}

\newcommand{\half}{\frac{1}{2}}
\newcommand{\cone}{\mathbf{1}_{C_1}}
\newcommand{\ctwo}{\mathbf{1}_{C_2}}
\newcommand{\cpone}{\mathbf{1}_{C'_1}}
\newcommand{\cptwo}{\mathbf{1}_{C'_2}}

\newcommand{\cgone}{\mathbf{1}_{G_1}}
\newcommand{\cgtwo}{\mathbf{1}_{G_2}}
\newcommand{\one}{\mathbf{1}}
\newcommand{\zero}{\mathbf{0}}
\newcommand{\rd}{\color{red}}
\newcommand{\bk}{\color{black}}

\begin{document}

\twocolumn[
\aistatstitle{A Theoretical Case Study of Structured Variational Inference for Community Detection}
\aistatsauthor{Mingzhang Yin\And Y. X. Rachel Wang \And  Purnamrita Sarkar }
\aistatsaddress{ University of Texas at Austin \And  University of Sydney \And University of Texas at Austin } 
]
\allowdisplaybreaks

\begin{abstract}
 Mean-field variational inference (MFVI) has been widely applied in large scale Bayesian inference. However, MFVI assumes independent distribution on the latent variables, which often leads to objective functions with many local optima, making optimization algorithms sensitive to initialization. In this paper, we study the advantage of  structured variational inference in the context of a simple two-class Stochastic Blockmodel. To facilitate theoretical analysis, the variational distribution is constructed to have a simple pairwise dependency structure on the nodes of the network.  We prove that, in a broad density regime and for general random initializations, unlike MFVI,  the estimated class labels  by structured VI converge to the ground truth with high probability, when the model parameters are known, estimated within a reasonable range or jointly optimized with the variational parameters. In addition, empirically we demonstrate structured VI is more robust compared with MFVI when the graph is sparse and the signal to noise ratio is low. The paper takes a first step towards quantifying the role of added dependency structure in variational inference for community detection.
\end{abstract}

\section{Introduction}
Variational inference (VI) is a widely used technique for approximating complex likelihood functions in Bayesian learning~\citep{Jordan:1999:VM,Blei:2003:LDA,Jaakkola:1999:IMF:308574.308663}, and is known for its computational scalability. VI reduces an intractable posterior inference problem to an optimization framework by imposing simpler dependence structure and is considered a popular alternative to Markov chain Monte Carlo (MCMC) methods.  Similar to the Expectation Maximization (EM) algorithm~\citep{dempster1977maximum}, VI works by the basic principle of constructing a tractable lower bound on the complete log-likelihood of a probabilistic model. One of the simplest forms of approximation is mean-field variational inference (MFVI),  where the variational lower bound, also known as ELBO, is computed using the expectation with respect to a product distribution over the latent variables\citep{Blei:2003:LDA, blei2006variational, hoffman2013stochastic}. 
Though VI has achieved great empirical success in probabilistic models, theoretical understanding of its convergence properties is still an open area of research. 

Theoretical studies of variational methods (and similar algorithms that involve iteratively maximizing a lower bound) have drawn significant attention recently (see \citep{balakrishnan2017statistical, xu2016global,  yan2017convergence, yi2014alternating, kwon2018global} for convergence properties of EM). %
For VI, the global optimizer of the variational lower bound is shown to be asymptotically consistent for a number of models including Latent Dirichlet Allocation (LDA)~\citep{Blei:2003:LDA} and Gaussian mixture models~\cite{pati2017statistical}. In~\cite{westling2015beyond} the connection between VI estimates and profile M-estimation is explored and asymptotic consistency is established. In practice, however, it is well known the algorithm is not guaranteed to reach the global optimum and the performance of VI often suffers from local optima~\cite{blei2017variational}. While in some models, convergence to the global optimum can be achieved with appropriate initialization ~\citep{wang2006convergence,awasthi2015some}, understanding convergence with general initialization and the influence of local optima is less studied with a few exceptions~\citep{xu2016global,ghorbani2018instability, mukherjee2018mean}.

In general, despite being computationally scalable, MFVI suffers from many stability issues including symmetry-breaking, multiple local optima, and sensitivity to initialization, which are consequences of the non-convexity of typical mean-field problems~\cite{wainwright2008graphical,jaakkola200110}. %
The independence assumption on latent variables also leads to the underestimation of posterior uncertainty \citep{blei2017variational, pmlr-v80-yin18b}. To address these problems, many studies suggest that modeling the latent dependency structure can expand the variational family under consideration and lead to larger ELBO and more stable convergence~\cite{xing2002generalized, hoffman2015structured, giordano2015linear, tran2015copula, ranganath2016hierarchical, pmlr-v80-yin18b, rezende2015variational, tran2017deep}. However, rigorous theoretical analysis with convergence guarantees in this setting remains largely underexplored.%

In this paper, we aim to study the effect of added dependency structure in a MFVI framework. Since the behavior of the log-likelihood of MFVI is well understood for the very simple two class, equal sized Stochastic Blockmodel (SBM)~\citep{mukherjee2018mean,zhang2017theoretical}, we propose to add a simple pairwise link structure to MFVI in the context of inference for SBMs. 
We study how added dependency structure can improve MFVI. In particular, we focus on how random initialization behave for VI with added structure.

The stochastic blockmodel (SBM)~\cite{holland1983stochastic} is a widely used network model for community detection in networks. There are a plethora of algorithms with theoretical guarantees for estimation for SBMs like Spectral methods~\citep{rohe2011spectral,coja2010graph}, semidefinite relaxation based methods~\citep{guedon2016community,perry2017semidefinite,amini2018semidefinite}, likelihood-based methods~\citep{amini2013pseudo}, modularity based methods~\citep{snijers1997mcmc,newman2004finding,bickel2009nonparametric}. Among these, likelihood-based methods remain important and relevant due to their flexibility in 
incorporating additional model structures. Examples include mixed membership SBM~\cite{airoldi2008mixed}, networks with node covariates~\cite{razaee2019matched}, and dynamic networks~\cite{matias2017statistical}. Among likelihood based methods,  VI provides a tractable approximation to the log-likelihood and is a scalable alternative to more expensive methods like Profile Likelihood~\citep{bickel2009nonparametric}, or MCMC based methods~\citep{snijers1997mcmc,newman2004finding}.   Computationally, VI was also shown to scale up well to very large graphs \cite{gopalan2013efficient}. 

On the theoretical front, \citep{bickel2013asymptotic} proved that the global optimum of MFVI behaves optimally in the dense degree regime. %
In terms of algorithm convergence, \citep{zhang2017theoretical} showed the batch coordinate ascent algorithm (BCAVI) for optimizing the mean-field objective has guaranteed convergence if the initialization is sufficiently close to the ground truth. \citep{mukherjee2018mean} fully characterized the optimization landscape and convergence regions of BCAVI for a simple two-class SBM with random initializations. It is shown that uninformative initializations can indeed converge to suboptimal local optima, demonstrating the limitations of the MFVI objective function.

Coming back to structured variational inference, it is important to note that, if one added dependencies  between the posterior of each node, the natural approximate inference method is the belief propagation (BP) algorithm \cite{pearl1982reverend, pearl2014probabilistic, wilinski2019detectability}. Based on empirical evidence, it has been conjectured in~\cite{decelle2011asymptotic} that BP is asymptotically optimal for a simple two-class SBM. In the sparse setting where phase transition occurs,~\cite{mossel2016belief} analyzed a local variant of BP and showed it is optimal given a specific initialization. 
In other parameter regions, rigorous theoretical understanding of BP, in particular, how adding dependence structure can improve convergence with general initializations is still an open problem.

Motivated by the above observations, we present a theoretical case study of structured variational inference for SBM. We emphasize here that our primary contribution \textit{does not} lie in proposing a new estimation algorithm that outperforms state-of-the-art methods; rather we use this algorithm as an example to understand the interplay between a non-convex objective function and an iterative optimization algorithm with respect to random initializations, and compare it with MFVI. We consider a two-class SBM with equal class size, an assumption commonly used in theoretical work~\citep{mossel2016belief, mukherjee2018mean} where the analysis for the simplest case is nontrivial. 

We study structured VI by introducing a simple pairwise dependence structure between randomly paired nodes. By carefully bounding the mean field parameters and their logits in each iteration using a combination of concentration and Littlewood-Offord type anti-concentration  arguments~\cite{erdos1945lemma},  we prove that in a broad density regime and under a fairly general random initialization scheme, the Variational Inference algorithm with Pairwise Structure (VIPS) can converge to the ground truth with probability tending to one, when the parameters are known, estimated within a reasonable range, or updated appropriately (Section~\ref{sec:main}). This is in contrast to MFVI, where convergence only happens for a narrower range of initializations. In addition, VIPS can escape from certain local optima that exist in the MFVI objective. These results highlight the theoretical advantage of the added dependence structure. Empirically, we demonstrate that VIPS is more robust compared to MFVI when the graph is sparse and the signal to noise ratio is low (Section~\ref{sec:exp}). We observe similar trends hold in more general models with unbalanced class sizes and more than two classes. We hope that our analysis for the simple blockmodel setting can shed light on theoretical analysis of algorithms with more general dependence structure such as BP.  %

The paper is organized as follows. Section~\ref{sec:prelim} contains the model definition and introduces VIPS. We present our theoretical results in Section~\ref{sec:main}. Finally in Section~\ref{sec:exp}, we demonstrate the empirical performance of VIPS in contrast to MFVI and other algorithms. We conclude with a discussion on possible generalizations, accompanied by promising empirical results in Section~\ref{sec:discuss}.

\section{Preliminaries and Proposed Work}
\label{sec:prelim}
\subsection{Preliminaries}
The stochastic block model (SBM)  is a generative network model with community structure. A $K$-community SBM for $n$ nodes is generated as follows: each node is assigned to one of the communities in $\{1,\dots, K\}$ according to a Multinomial distribution with parameter $\pi$. These memberships are represented by $U\in \{0,1\}^{n\times K}$, where each row follows an independent Multinomial $(1;\pi)$ distribution. We have $U_{ik}=1$ if node $i$ belongs to community $k$ and $\sum_{k=1}^{K}U_{ik}=1$. Given the community memberships, links between pairs of nodes are generated according to the entries in a $K\times K$ connectivity matrix $B$. That is, if $A$ denotes the $n\times n$ binary symmetric adjacency matrix, then, for $i\neq j$,
\ba{
P(A_{ij}=1 | U_{ik}=1, U_{j\ell}=1) = B_{k\ell}.
\label{eq:likelihood_0}
}
We consider undirected networks, where both $B$ and $A$ are symmetric. Given an observed $A$, the goal is to infer the latent community labels $U$ and the model parameters $(\pi, B)$. Since the data likelihood $P(A;B,\pi)$ requires summing over $K^n$ possible labels, approximations such as MFVI are often needed to produce computationally tractable algorithms. 

Throughout the rest of the paper, we will use $\one_n$ to denote the all-one vector of length $n$. When it is clear from the context, we will drop the subscript $n$. Let $I$ be the identity matrix and $J=\mathbf{1}\mathbf{1}^T$.  $\mathbf{1}_C$ denotes a vector where the $i$-th element is $1$ if $i \in C$ and 0 otherwise, where $C$ is some index set. Similar to~\cite{mukherjee2018mean}, we consider a two-class SBM with equal class size, where $K=2$, $\pi=1/2$, and $B$ takes the form $B_{11}=B_{22} = p$, $B_{12}=B_{21}=q$, with $p>q$. We denote the two true underlying communities by $G_1$ and $G_2$, where $G_1, G_2$ form a partition of $\{1,2,\ldots,n\}$ and  $|G_1|=|G_2|$. (For convenience, we assume $n$ is even.) As will become clear, the full analysis of structured VI in this simple case is highly nontrivial.

\subsection{Variational inference with pairwise structure (VIPS)}

The well-known MFVI approximates the likelihood by assuming a product distribution over the latent variables. In other words, the posterior label distribution of the nodes is assumed to be independent in the variational distribution. To investigate how introducing dependence structure can help with the inference, we focus on a simple setting of linked pairs which are independent of each other. %
To be concrete, we randomly partition the $n$ nodes into two sets: $P_1 = \{z_1,\cdots,z_m\}$, $P_2 = \{y_1,\cdots,y_m\}$, with $m = n/2$. Here $z_k, y_k\in\{1, \dots, n\}$ are the node indices. In our structured variational distribution, we label pairs of nodes $(z_k, y_k)$ using index $k\in\{1,\ldots,m\}$ and assume there is dependence within each pair.
The corresponding membership matrices for $P_1$ and $P_2$ are denoted by $Z$ and $Y$ respectively, which are both $m \times 2$ sub-matrices of the full membership matrix $U$. More explicitly, the $k^{th}$ row of matrix $Z$ encodes the membership of node $z_k$ in $P_1$, and similarly for $Y$.
For convenience, we permute both the rows and columns of $A$ based on the node ordering in $P_1$ followed by that in $P_2$ to create a partitioned matrix: 
$A=\left[
\begin{array}{c|c}
A^{zz} & A^{zy} \\
\hline
A^{yz} & A^{yy}
\end{array}
\right]$,
where each block is an $m \times m$ matrix. Given the latent membership variable $(Z, Y)$, by Eq.~\eqref{eq:likelihood_0} the likelihood of $A$ is given by
\ba{
& \textstyle  P(A^{zz}_{ij}  | Z,B) = \prod_{a,b} [B_{ab}^{A_{ij}^{zz}}(1-B_{ab})^{1-A_{ij}^{zz}}]^{Z_{ia}Z_{jb}} \notag \\
&\textstyle  P(A^{zy}_{ij}  | Y,Z,B) = \prod_{a,b} [B_{ab}^{A_{ij}^{zy}}(1-B_{ab})^{1-A_{ij}^{zy}}]^{Z_{ia}Y_{jb}}\notag \\
&\textstyle  P(A^{yy}_{ij}  | Y,B) = \prod_{a,b} [B_{ab}^{A_{ij}^{yy}}(1-B_{ab})^{1-A_{ij}^{yy}}]^{Y_{ia}Y_{jb}}
\label{eq:likelihood}
}
where $a,b \in \{1,2\}$ and $A^{zy} = (A^{yz})^T$. %

A simple illustration of the partition and how ordered pairs of nodes are linked to incorporate dependence is given in Figure~\ref{fig:demo}, where the the true underlying communities $G_1$ and $G_2$ are shaded differently. After the partition, we have $m$ pairs of linked nodes indexed from 1 to $m$. For convenience of analysis, we define the following sets for these pairs of linked nodes, as illustrated in Figure~\ref{fig:demo}. %

Define $C_1$, ($C_1'$) as the set of indices $i$ of pairs $(z_i,y_i)$ with $z_i\in G_1$, ($y_i\in G_1$). Similarly, $C_2$, ($C_2'$) is the set of indices of pairs $(z_i,y_i)$ with $z_i\in G_2$, ($y_i\in G_2$). %
We will also make use of the sets $C_{ab} \coloneqq C_{a} \cap C'_{b}$, where $a,b \in \{1,2\}$. In Figure~\ref{fig:demo}, these sets correspond to different combinations of shading, i.e. community memberships, of the linked pairs, e.g. $C_{12}$ is the index set of pairs $(z_i,y_i)$ with $z_i\in G_1,y_i\in G_2$. %

We define the variational distribution for the latent membership matrix $(Z,Y)$ as $Q(Z,Y)$, which we assume takes the form
\begin{align}
   Q(Z,Y) = \prod_{i=1}^m Q(Z_i,Y_i), 
\end{align}
where $Z_i$ denotes the $i^{th}$ row of $Z$, and $Q(Z_i,Y_i)$ is a general categorical distribution with variational parameters defined as follows. 
 \begin{align*}
   \psi_i^{cd} \coloneqq Q(Z_{i,c+1}=1, Y_{i,d+1}=1),
\end{align*}
for $i\in\{1, \dots, m\}, c,d\in\{0,1\}$.
 This allows one to encode more dependence structure between the posteriors at different nodes than vanilla MFVI, since we allow for dependence within each linked pair of nodes while keeping independence between different pairs. We define the marginal probabilities as:
  \begin{align}
    \phi_i &\coloneqq Q(Z_{i1}=1)=\psi_i^{10}+\psi_i^{11} \nonumber\\
    \xi_i &\coloneqq Q(Y_{i1}=1)=\psi_i^{01}+\psi_i^{11}.
    \label{eq:marginal}
  \end{align} 
\begin{figure}[ht]
\includegraphics[width=8cm]{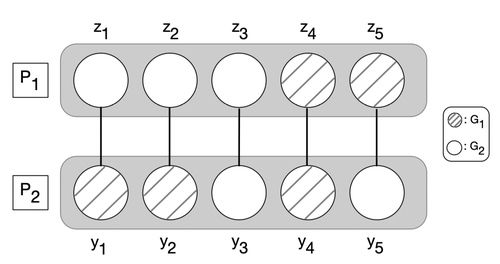}
\vspace{-5pt}
\caption{
  An illustration of a partition for $n=10$. The shaded nodes belong to community $G_1$ and unshaded nodes belong to community $G_2$. The nodes are randomly partitioned into two sets $P_1$ and $P_2$, and pairs of nodes are linked from index 1 to $m$. Dependence structure within each linked pair is incorporated into the variational distribution $Q(Z,Y)$.  For this partition and pair linking, $C_1 = \{4,5\}$, $C_2 = \{1,2,3\}$, $C'_1 = \{1,2,4\}$, $C'_2 = \{3,5\}$; $C_{11} = \{4\}$, $C_{12} = \{5\}$, $C_{21} = \{1,2\}$, $C_{22} = \{3 \}$.
    } \label{fig:demo}
\end{figure}
Next we derive the ELBO on the data log-likelihood $\log P(A)$ using $Q(Z,Y)$. For pairwise structured variational inference (VIPS), ELBO takes the form
\bas{
\cL(Q;\pi,B) =& \E_{Z,Y \sim Q(Z,Y)} \log  P(A | Z, Y) \\
-& \kl(Q(Z,Y)||P(Z,Y)),
}
where $P(Z,Y)$ is the probability of community labels from SBM and follows independent Bernoulli $(\pi)$ distribution, $\kl(\cdot || \cdot)$ denotes the usual Kullback–Leibler divergence between two distributions. Using the likelihood in Eq.~\eqref{eq:likelihood}, the ELBO becomes
\ba{ \notag
\cL(Q;\pi,B) =&  \half \E_Q \sum_{i\neq j,a,b} Z_{ia}Z_{jb}(A_{ij}^{zz}\alpha_{ab}+f(\alpha_{ab}))\\ \notag
+& \half \E_Q \sum_{i\neq j,a,b}Y_{ia}Y_{jb}(A_{ij}^{yy}\alpha_{ab}+f(\alpha_{ab})) \\ \notag
+& \E_Q \sum_{i\neq j,a,b} Z_{ia}Y_{jb}(A_{ij}^{zy}\alpha_{ab}+f(\alpha_{ab}))\\ \notag
+&\E_Q \sum_{i,a,b}Z_{ia}Y_{ib}(A_{ii}^{zy}\alpha_{ab}+f(\alpha_{ab})) \\
-&   \sum_{i=1}^m \kl(Q(z_i,y_i) || P(z_i)P(y_i)),
\label{eq:elbo}
}
where $\alpha_{ab} = \log(B_{ab}/(1-B_{ab}))$ and $f(\alpha)=-\log(1+e^{\alpha})$. The KL regularization term can be computed as
\ba{   
 &\textstyle \kl(Q(z_i,y_i) || P(z_i)P(y_i))  \notag \\
 =& \textstyle \psi_i^{00}\log\frac{\psi_i^{00}}{(1-\pi)^2} +\psi_i^{01}\log\frac{\psi_i^{01}}{\pi(1-\pi)}  \notag  \\
 &~~~~~~~~~~~~~~~~~~+\psi_i^{10}\log\frac{\psi_i^{10}}{\pi(1-\pi)}  +\psi_i^{11}\log\frac{\psi_i^{11}}{\pi^2} \notag  \\ 
=&\textstyle \sum_{0\leq c,d \leq 1} \psi_i^{cd}\log(\psi_i^{cd})/(\pi^c\pi^d(1-\pi)^{1-c}(1-\pi)^{1-d}).
\label{eq:elbo-kl}
}

Our goal is to maximize $\cL(Q;\pi,B)$ with respect to the variational parameters $\psi_i^{cd}$ for $1\leq i \leq m$. Since $\sum_{c,d}\psi_i^{cd}=1$ for each $i$, it suffices to consider $\psi_i^{10}, \psi_i^{01}$ and $\psi_i^{11}$. By taking derivatives, we can derive a batch coordinate ascent algorithm for updating $\psi^{cd}=(\psi_1^{cd}, \dots, \psi_m^{cd})$. Detailed calculation of the derivatives can be found in Section~\ref{sec:elboderiv} of the Appendix. Recall that $\pi = \half$. Also, define 
\ba{
&t \coloneqq \half\log \frac{p/(1-p)}{q/(1-q)}\qquad \lambda  \coloneqq  \frac{1}{2t}\log \frac{1-q}{1-p},\\ 
&\theta^{cd} \coloneqq \log \frac{\psi^{cd}}{1-\psi^{01}-\psi^{10}-\psi^{11}},
}
where $\theta^{cd}$ are logits, $c,d\in \{0,1\}$ and all the operations are defined \textit{element-wise.}

Given the model parameters $p,q$, the current values of $\psi^{cd}$ and the marginals $\phi=\psi^{10}+\psi^{11}$, $\xi=\psi^{01}+\psi^{11}$ as defined in Eq.~\eqref{eq:marginal},  the updates for $\theta^{cd}$ are given by:
\ba{
\theta^{10} =&  \textstyle 4t[A^{zz} - \lambda(J-I)](\phi - \half \mathbf{1}_m) \notag \\
& + 4t[A^{zy}-\lambda(J-I)-\text{diag}(A^{zy})](\xi - \half \mathbf{1}_m)   \notag\\
& - 2t(\text{diag}(A^{zy})-\lambda I) \mathbf{1}_m,   \label{eq:thetaoz} \\ 
\theta^{01} =& \textstyle 4t[A^{yy} - \lambda(J-I)](\xi - \half \mathbf{1}_m)   \notag \\
& + 4t[A^{yz}-\lambda(J-I)-\text{diag}(A^{yz})](\phi - \half \mathbf{1}_m) \notag  \\
& -2t(\text{diag}(A^{yz})-\lambda I) \mathbf{1}_m, \label{eq:thetazo}  \\  
\theta^{11} =& \textstyle 4t[A^{zz} - \lambda(J-I)](\phi - \half \mathbf{1}_m) \notag \\
&+ 4t[A^{zy}-\lambda(J-I) -\text{diag}(A^{zy})](\xi - \half \mathbf{1}_m) \notag \\ 
& +  \textstyle 4t[A^{yy} - \lambda(J-I)](\xi - \half \mathbf{1}_m) \notag \\
&+ 4t[A^{yz}-\lambda(J-I)-\text{diag}(A^{yz})](\phi - \half \mathbf{1}_m).   \label{eq:thetaoo}
}
Given $\theta^{cd}$, we can update the current values of $\psi^{cd}$ and the corresponding marginal probabilities $\phi$, $\xi$ using element-wise operations as follows:
\begin{align}
&\textstyle \psi^{cd} = \dfrac{e^{\theta^{cd}}}{1+e^{\theta^{01}}+e^{\theta^{11}}+ e^{\theta^{10}}}, ~~ u \coloneqq (\phi, \xi)\notag \\
&\textstyle \phi = \dfrac{e^{\theta^{10}}+e^{\theta^{11}}}{1+e^{\theta^{10}}+e^{\theta^{01}}+e^{\theta^{11}}}, ~~\xi = \dfrac{e^{\theta^{01}}+e^{\theta^{11}}}{1+e^{\theta^{10}}+e^{\theta^{01}}+e^{\theta^{11}}},
\label{eq:phixi}
\end{align}
where $(c,d)=(1,0),(0,1), (1,1)$. The marginal probabilities are concatenated as $u = (\phi, \xi)\in [0,1]^n$. Thus $u$ can be interpreted as the estimated posterior membership probability of all the nodes. 

Since $\theta^{cd}$ determines $\psi^{cd}$ in the categorical distribution and $u$ represents the corresponding marginals, one can think of $\theta^{cd}$ and $u$ as the local and global parameters respectively. It has been empirically shown that the structured variational methods can achieve better convergence property by iteratively updating the local and global parameters \citep{Blei:2003:LDA, hoffman2013stochastic, hoffman2015structured}. In the same spirit, in the full optimization algorithm, we update the parameters $\theta^{cd}$ and $u$
iteratively by
 \eqref{eq:thetaoz}--\eqref{eq:phixi}, following the order
\ba{
\theta^{10} \to u \to \theta^{01} \to u \to \theta^{11} \to u \to  \theta^{10}  \cdots. 
\label{eq:update_rule}
}
We call a full update of all the parameters $\theta^{10}, \theta^{01}, \theta^{11}, u$  in \eqref{eq:update_rule} as one \textit{meta iteration} which consists of three inner iterations of $u$ updates.  We use $u_j^{(k)}$ ($j=1,2,3$) to denote the update in the $j$-th iteration of the $k$-th meta iteration, and $u^{(0)}$ to denote the initialization.   Algorithm~\ref{alg:M1} gives the full algorithm when the model parameters are known.

\begin{algorithm}[H] 
\small{
\SetKwData{Left}{left}\SetKwData{This}{this}\SetKwData{Up}{up}
\SetKwFunction{Union}{Union}\SetKwFunction{FindCompress}{FindCompress}
\SetKwInOut{Input}{input}\SetKwInOut{Output}{output}
\Input{
Adjacency matrix $A \in \{0,1\}^{n \times n}$, model parameter $p, q, \pi=1/2$. 
}
\Output{The estimated node membership vector $u$. }
\BlankLine
Initialize the elements of $u$ i.i.d. from an arbitrary distribution $f_{\mu}$ defined on $[0,1]$ with mean $\mu$. Initialize $\theta^{10} = \theta^{01} = \theta^{11} = \mathbf{0}$;

Randomly select $n/2$ nodes as $P_1$  and the other $n/2$ nodes as $P_2$;

\While{not converged}{
Update $\theta^{10}$ by \eqref{eq:thetaoz}.

Update $u = (\phi,\xi)$ by \eqref{eq:phixi}

Update $\theta^{01}$ by \eqref{eq:thetazo}.

Update $u = (\phi,\xi)$ by \eqref{eq:phixi}

Update $\theta^{11}$ by \eqref{eq:thetaoo}.

Update $u = (\phi,\xi)$ by \eqref{eq:phixi}

}
\caption{Variational Inference with Pairwise Structure (VIPS)}
\label{alg:M1}}
\end{algorithm}

\begin{remark}
So far we have derived the updates and described the optimization algorithm when the true parameters $p, q$ are known. When they are unknown, they can be updated jointly with the variational parameters after each meta iteration as 
\ba{
p =& \frac{  \splitdfrac{(\one_n-u)^T A (\one_n-u) + u^T A u}{ + 2(\one_m-\psi^{10}-\psi^{01})^T \text{diag}(A^{zy})\one_m}}{\splitdfrac{(\one_n-u)^T (J-I) (\one_n-u)} {+ u^T (J-I) u  + 2(\one_m-\psi^{10}-\psi^{01})^T \one_m}} \notag \\
q =& \frac{(\one_n-u)^T A u + (\psi^{10}+\psi^{01})^T \text{diag}(A^{zy})\one_m}{(\one_n-u)^T (J-I) u_n   + (\psi^{10}+\psi^{01})^T \one_m } 
\label{eq:hatpq}
}
Although it is typical to update $p,q$ and $u$ jointly, as shown in \cite{mukherjee2018mean}, analyzing MFVI updates with known parameters can shed light on the convergence behavior of the algorithm. Initializing $u$ randomly while jointly updating $p,q$ always leads MFVI to an uninformative local optima. For this reason, in what follows we will analyze Algorithm \ref{alg:M1} in the context of both fixed and updating parameters $p,q$. 
\label{rem:update_pq}
\end{remark}

\iffalse
\rd
\begin{remark}
In addition to updating the parameters, the partitions $P_1$, $P_2$ can be updated too as the algorithm starts to find divisions of the nodes closer to the ground truth. A natural extension of the VIPS algorithm is to update $P_1$, $P_2$ at the end of each meta iteration as follows: given the current $u$, let $I_1=\{i:u(i)<=0.5\}$, $I_2=\{i:u(i)>0.5\}$, where $u(i)$ is the $i$-th entry of $u$. If $I_1$ and $I_2$ have different sizes, we randomly move some indices from the bigger set to the smaller set to make the sizes equal. These sets will be the updated $P_1$ and $P_2$. We call this algorithm VIPS-plus. Due to the intricacy involved in the theoretical analysis of VIPS, we leave the analysis of the VIPS-plus for future work. We examine the empirical performance of VIPS-plus in Section~\ref{sec:exp}.
\label{rem:update_partition}
\end{remark}
\bk
\fi

\section{Main results}
\label{sec:main}
In this section, we present theoretical analysis of the algorithm in three settings: (i) When the parameters are set to the true model parameters $p,q$; (ii) When the parameters are not too far from the true values, and are held fixed throughout the updates; (iii) Starting from some reasonable guesses of the parameters, they are jointly updated with latent membership estimates.

In the following analysis, we will frequently use the eigen-decomposition of the expected adjacency matrix $P = \E[A|U] =  \frac{p+q}{2} \mathbf{1}_n \mathbf{1}_n^T + \frac{p-q}{2} v_2v_2^T  - p I$  where $v_2 = (v_{21}, v_{22})^T = (\cone - \ctwo, \cpone - \cptwo)^T$ is the second eigenvector. Since the second eigenvector is just a shifted and scaled version of the membership vector, the projection $|\ip{u}{v_2}|$ is equivalent to the $\ell_1$ error from true label $z^*$ (up-to label permutation) by $\norm{u - z^*}_1 = m - |\ip{u}{v_2}|$. 
We consider the parametrization $p \asymp q \asymp \rho_n$, where the density $\rho_n\to 0$ at some rate and $p-q=\Omega(\rho_n)$.

When the true parameters $p,q$ are known, it has been shown \citep{purna2019mf} that without dependency structure, MFVI with random initializations converges to the stationary points with non-negligible probability. When the variational distribution has a simple pairwise dependency structure as VIPS, we show a stronger result. To be concrete, in this setting, we establish that convergence happens with probability approaching 1. In addition, unlike MFVI, the convergence holds for general random initializations. We will first consider the situation when $u^{(0)}$ is initialized from a distribution centered at $\mu = \half$ and show the results for $\mu \neq \half$ in Corollary \ref{cor:not-half}.

\begin{theorem}[Sample behavior for known parameters]
Assume $ \theta^{10}, \theta^{01}, \theta^{11}$ are initialized as $ \zero$  and the elements of $u^{(0)}=(\phi^{(0)},\xi^{(0)})$ are initialized i.i.d. from $\text{Bernoulli}(\frac{1}{2})$. When $p \asymp q \asymp \rho_n$, $p-q=\Omega(\rho_n)$,  and $\sqrt{n}\rho_n = \Omega(\log(n))$, Algorithm~\ref{alg:M1} converges to the true labels asymptotically after the second meta iteration, in the sense that
\bas{
\|u_3^{(2)} - z^*\|_1 =  n\exp(-\Omega_P(n\rho_n)))
}
$z^*$ are the true labels with $z^* = \cgone$ or $\cgtwo$. The same convergence holds for all the later iterations.
\label{thm: convergence}
\end{theorem}
\begin{remark}
It is important to note that there are many algorithms (see ~\cite{abbe2017community} for a survey) which recover the memberships exactly in this regime. We do not compare our theoretical results with those or to well known thresholds for exact recovery~\cite{abbe2015exact}, because our goal is not to design a new algorithm with an improved theoretical guarantees. Rather, we show that by introducing the simplest possible pairwise dependence structure, variational inference for a simple setting of a SBM improves over MFVI which has no such structure. The density regime simply makes the analysis somewhat easier.
\end{remark}
\begin{proof} We provide a proof sketch here and defer the details to Section~\ref{sec:proofmain} of the Appendix. We assume for the first six iterations, we randomly partition $A$ into six $A^{(i)}, i=0, \dots, 5$ by assigning each edge to one of the six subgraphs with equal probability. For the later iterations, we can use the whole graph $A$. Then $A^{(i)}$'s are independent with population matrix $P/6$. Although not used in Algorithm~\ref{alg:M1}, the graph splitting is a widely used technique for theoretical convenience \cite{mcsherry2001spectral,chaudhuri2012} and allows us to bound the noise in each iteration more easily. The main arguments involve lower bounding the size of the projection $|\ip{u}{v_2}|$ in each iteration as it increases towards $n/2$, at which point the algorithm achieves strong consistency. For ease of exposition, we will scale everything by $6$ so that $p,q,\lambda$ correspond to the parameters for the full un-split matrix $P$. This does not affect the analysis in any way.

 In each iteration, we decompose the intermediate $\theta^{10}, \theta^{01}, \theta^{11}$ into blockwise constant signal and random noise using the spectral property of the population matrix $P$. As an illustration, in the first meta iteration, 
we write the update in \eqref{eq:thetaoz}--\eqref{eq:thetaoo} as signal plus noise,
\bas{
&\theta_i^{10} = 4t(s_1 \cone + s_2 \ctwo +  r_i^{(0)}) \\
&\theta_i^{01} = 4t(x_1 \cpone + x_2 \cptwo +  r_i^{(1)}) \\
&\theta_i^{11} = 4t(y_1 \cone + y_2 \ctwo +y_1 \cpone + y_2 \cptwo +  r_i^{(2)})
}
where $t$ is a constant and the noise has the form
\ba{
r^{(i)} = R^{(i)}(u_j^{(k)} - \half\one)
\label{eq:r_i}
}
for appropriate $j,k$, where $R^{(i)}$ arises from the sample noise in the adjacency matrix. We handle the noise from the first iteration $r^{(0)}$ with a Berry-Esseen bound conditional on $u^{(0)}$, and the later $r^{(i)}$ with a uniform bound. The blockwise constant signals $s_1, x_1, y_1$ are updated as $( \frac{p+q}{2} -\lambda)(\ip{u}{\mathbf{1}_n}-m)  + (  \frac{p-q}{2}) \ip{u}{v_2}$ and $s_2, x_2, y_2$ are updated as $( \frac{p+q}{2} -\lambda) (\ip{u}{\mathbf{1}_n}-m)  - (  \frac{p-q}{2}) \ip{u}{v_2}$. As $\ip{u}{v_2}$ increases throughout the iterations, the signals become increasingly separated for the two communities. Using Littlewood-Offord type anti-concentration, we show in the first meta iteration, 
\ba{
&\ip{u_1^{(1)}}{v_{2}} = \Omega_P(n \sqrt{\rho_n}), \quad\ip{u_1^{(1)}}{\one} - m= 0 \notag\\
&\ip{u_2^{(1)}}{v_2} \geq \frac{n}{8}  - o_P(n), \quad \ip{u_2^{(1)} }{\one} - m = 0 \notag \\
&\ip{u_3^{(1)}}{v_2} \geq  \frac{1}{4}n  + o_P(n), \notag\\
& -\frac{n}{8}-o_P(n) \leq \ip{u_3^{(1)}}{\one}- m \leq \frac{n}{4}  + o_P(n)
\label{eq:signal}
} 
After the second meta iteration we have %
\begin{equation}\label{eq:conditions}
\begin{gathered}
s_1^{(2)}, x_1^{(2)} ,y_1^{(2)} = \Omega_P(n \rho_n), \\
s_2^{(2)},  x_2^{(2)}, y_2^{(2)} = - \Omega_P(n  \rho_n); \\
2y_1^{(2)} - s_1^{(2)} = \Omega_P(n \rho_n),\\
2y_1^{(2)} - x_1^{(2)} = \Omega_P(n \rho_n); \\
s_1^{(2)} - (y_1^{(2)} + y_2^{(2)}) = \Omega_P(n\rho_n), \\
x_1^{(2)} - (y_1^{(2)} + y_2^{(2)})  = \Omega_P(n\rho_n);
\end{gathered}
\end{equation}
Plugging \eqref{eq:conditions} to \eqref{eq:phixi}, we have the desired convergence after the second meta iteration. 
\end{proof}

The next corollary shows the same convergence holds when we use a general random initialization not centered at $1/2$. In contrast, MFVI converges to stationary points $\mathbf{0}_n$ or $\one_n$ with such initializations.
\begin{cor}
Assume the elements of $u^{(0)}$ are i.i.d. sampled from a distribution with mean $\mu \neq 0.5$. Under the  conditions in Theorem~\ref{thm: convergence}, applying Algorithm~\ref{alg:M1} with known $p,q$, we have $\|u_1^{(3)} - z^*\|_1 = n\exp(-\Omega_P(n\rho_n)))$. The same order holds for all the later iterations. 
\label{cor:not-half}
\end{cor}
The proof relies on showing after the first iteration, $u_1^{(1)}$ behaves like nearly independent $\text{Bernoulli}(\frac{1}{2})$, the details of which can be found in Appendix~\ref{sec:proofmain}. 

The next proposition focuses on the behavior of special points in the optimization space for $u$. In particular, we show that Algorithm~\ref{alg:M1} enables us to move away from the stationary points $\zero_n$ and $\one_n$, whereas in MFVI, the optimization algorithm gets trapped in these stationary points \cite{mukherjee2018mean}. 

\begin{proposition} [Escaping from stationary points]  \hfill 
\begin{enumerate}[(i)]
\item 
 $(\psi^{00}, \psi^{01},\psi^{10}, \psi^{11}) = (\one,\mathbf{0},\mathbf{0},\mathbf{0})$,  $(\mathbf{0},\mathbf{0},\mathbf{0},\one)$ (these vectors are $m$-dimensional) are the stationary points of the pairwise structured ELBO when $p,q$ are known, which maps to $u = \mathbf{0}_n$ and $\one_n$ respectively. 
\item With the updates in Algorithm \ref{alg:M1}, when $u^{(0)} = \mathbf{0}_n$, $\one_n$, %
VIPS converges to the true labels with $\|u_1^{(3)} - z^*\|_1 = n\exp(-\Omega_P(n\rho_n)))$.
\end{enumerate}
\label{prop:stat_pt}
\end{proposition}
The above results requires knowing the true $p$ and $q$. The next corollary shows that, even if we do not have access to the true parameters, as long as some reasonable estimates can be obtained, the same convergence as in Theorem \ref{thm: convergence} holds thus demonstrating robustness to misspecified parameters. Here we hold the parameters fixed and only update $u$ as in Algorithm ~\ref{alg:M1}.
\begin{proposition}
[Parameter robustness]
	\label{cor:pq_noise}
	If we replace true $p, q$ with some estimation $\hat{p}, \hat{q}$ in Algorithm ~\ref{alg:M1},
	the same conclusion as in Theorem~\ref{thm: convergence} holds  if
    \begin{center}
	\begin{inparaenum}
		\item $\frac{p+q}{2} > \hat{\lambda}$, \quad
		\item $\hat{\lambda} - q= \Omega(\rho_n)$, \quad
		\item $\hat{t}=\Omega(1) $.
	\end{inparaenum}
	\end{center}
	where $\hat{t} = \half\log \dfrac{\hat{p}/(1-\hat{p})}{\hat{q}/(1-\hat{q})}$, $\hat{\lambda} =\frac{1}{2\hat{t}}\log \dfrac{1-\hat{q}}{1-\hat{p}}$.
\end{proposition}
When $\hat{p}, \hat{q} \asymp \rho_n$, we need $\hat{p} - \hat{q} = \Omega(\rho_n)$ and $\hat{p}, \hat{q}$ not too far from the true values to achieve convergence. The proof is deferred to the Appendix. 

Finally, we consider updating the parameters jointly with $u$ (as explained in Remark \ref{rem:update_pq}) by first initializing the algorithm with some reasonable $p^{(0)}, q^{(0)}$.
\begin{theorem}
[Updating parameters and $u$ simultaneously]
\label{cor:update_pq}
Suppose we initialize with some estimates of true $(p,q)$ as $\hat{p}=p^{(0)}$, $\hat{q}=q^{(0)}$ satisfying the conditions in Proposition \ref{cor:pq_noise} and apply two meta iterations in Algorithm \ref{alg:M1} to update $u$ before updating $\hat{p}=p^{(1)}, \hat{q}=q^{(1)}$. After this, we alternate between updating $u$ and the parameters after each meta iteration. Then
\bas{
& p^{(1)} = p+O_P(\sqrt{\rho_n}/n), ~~~q^{(1)} = q+O_P(\sqrt{\rho_n}/n), \\
& \|u_3^{(2)}-z^*\|_1 = n\exp(-\Omega(n\rho_n)),
}
and the same holds for all the later iterations.
\end{theorem}

\section{Experiments}
\label{sec:exp}
In this section, we present some numerical results. In Figures~\ref{fig:converge} to~\ref{fig:snr} we show the effectiveness of VIPS in our theoretical setting of two equal sized communities. In Figures~\ref{fig:snr2} (a) and (b) we show that empirically the advantage of VIPS holds even for unbalanced community sizes and $K>2$. Our goal is two-fold: (i) we demonstrate that the empirical convergence behavior of VIPS coincides well with our theoretical analysis in Section~\ref{sec:main}; (ii) in practice VIPS has superior performance over MFVI in both the simple setting we have analyzed and more general settings, thus confirming the advantage of the added dependence structure. For the sake of completeness, we also include comparisons with other popular algorithms, even though it is not our goal to show VIPS outperforms these methods.

In Figure \ref{fig:converge}, we compare the convergence property of VIPS with MFVI for initialization from independent Bernoulli's with means $\mu=0.1,0.5$, and $0.9$.  We randomly generate a graph with $n = 3000$ nodes with parameters $p_0= 0.2, q_0 = 0.01$ and show results from 20 random trials. %
We plot $\min(\|u-z^*\|_1, \|u-(\one-z^*)\|_1)$, or the $\ell_1$ distance of the estimated label $u$ to the ground truth $z^*$  on the $Y$ axis versus the iteration number on the $X$ axis. In this experiments, both VIPS and MFVI were run with the true $p_0,q_0$ values.
As shown in Figure \ref{fig:converge}, when $\mu = \half$, VIPS converges to $z^*$ after two meta iterations (6 iterations) for all the random initializations. In contrast, for MFVI, a fraction of the random initializations converge to $\mathbf{0}_n$ and $\one_n$. When $\mu \neq \half$, VIPS converges to the ground truth after three meta iterations, whereas MFVI stays at the stationary points $\mathbf{0}_n$ and $\one_n$. This is consistent with our theoretical results in Theorem~\ref{thm: convergence} and Corollary~\ref{cor:not-half}, and those in~\cite{mukherjee2018mean}.

\begin{figure}[ht]
\centering
\begin{tabular}{ccc}
\hspace{-0.5cm}
 \includegraphics[width=0.155\textwidth]{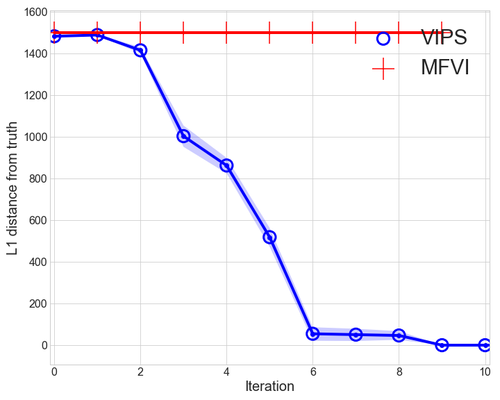}& 
 \hspace{-0.5cm}
 \includegraphics[width=0.155\textwidth]{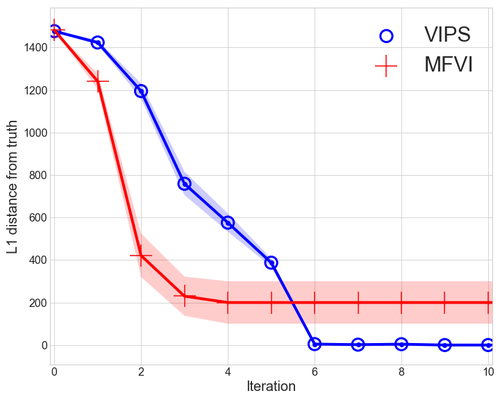}&
 \hspace{-0.5cm}
  \includegraphics[width=0.155\textwidth]{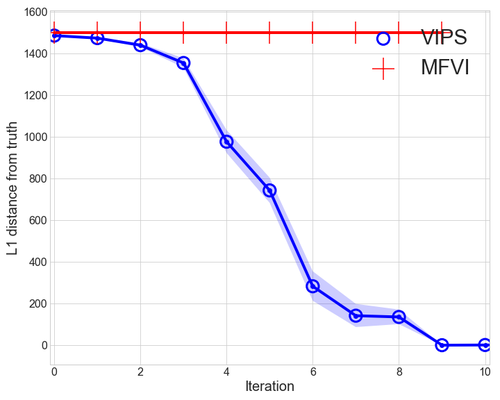}
 \\
\end{tabular}
 \caption{
$\ell_1$ distance from ground truth ($Y$ axis) vs. number of iterations ($X$ axis). The line is the mean of 20 random trials and the shaded area shows the standard deviation. %
$u$ is initialized from i.i.d. Bernoulli with mean $\mu = 0.1, 0.5, 0.9$ from the left to right. 
}
\label{fig:converge}
 \end{figure}

In Figure \ref{fig:heatmap}, we show when the true $p,q$ are unknown, the dependence structure makes the algorithm more robust to estimation errors in $\hat{p}, \hat{q}$.  The heatmap represents the normalized mutual information (NMI) \citep{romano2014standardized} between $u$ and $z^*$, with $\hat{p}$ on the $X$ axis and $\hat{q}$ on the $Y$ axis. We only examine pairs with $\hat{p}>\hat{q}$. Both VIPS and MFVI were run with $\hat{p}$ and $\hat{q}$, which were held fixed and differ from the true values to varying extent. The dashed line represents the true $p,q$ used to generate the graph. For each $\hat{p}, \hat{q}$ pair, the mean NMI for 20 random initializations from i.i.d Bernoulli($\half$) is shown.  VIPS recovers the ground truth in a wider range of $\hat{p}, \hat{q}$ values than MFVI. We show in Section~\ref{sec:expaddl} of the Appendix that similar results also hold for $K=2$ with unbalanced community sizes.

\vspace{-3pt}
\begin{figure}[ht]
\centering
\begin{tabular}{cc}
(a)~MFVI& (b)~VIPS \\
 \includegraphics[width=0.23\textwidth]{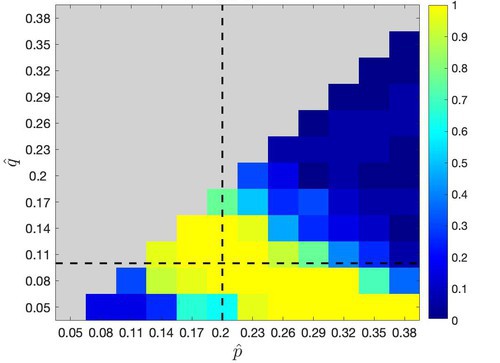}& 
\includegraphics[width=0.23\textwidth]{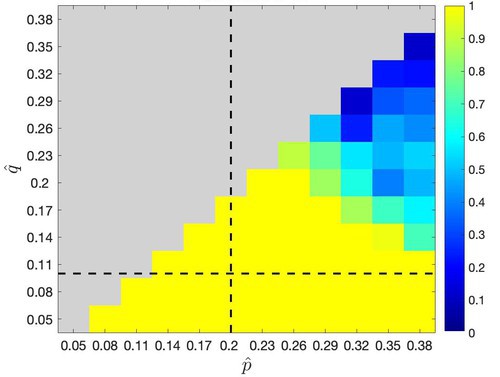} 
\end{tabular}
\caption{
NMI averaged over 20 random initializations for each $\hat{p}$, $\hat{q}$ ($\hat{p} > \hat{q}$). The true parameters are $(p_0, q_0) = (0.2, 0.1)$, $\pi = 0.5$ and $n=2000$. The dashed lines indicate the true parameter values.
}
 \label{fig:heatmap}
\end{figure}

In Figure \ref{fig:snr}, we compare VIPS with MFVI under different network sparsities and signal-to-noise ratios (SNR) as defined by $r_0=p_0/q_0$. For the sake of completeness, we also include two other popular algorithms, Belief Propagation (BP)~\cite{krzkalaBP2011} and Spectral Clustering~\cite{rohe2011spectral}. We plot the mean and standard deviation of NMI for 20 random trials in each setting. In each trial, to meet the conditions in Theorem~\ref{cor:update_pq}, we started VIPS with $\hat{p}$ equal to the average degree of $A$, and $\hat{q} = \hat{p}/r_0$. $\hat{p}$ and $\hat{q}$ were updated alternatingly with $u$ according to Eq.~\eqref{eq:hatpq} after three meta iterations in Algorithm~\ref{alg:M1}, a setting similar to that of Theorem~\ref{cor:update_pq}. 
 
 In Figure \ref{fig:snr}-(a), the average expected degree is fixed as the SNR $p_0/q_0$ increases on the $X$ axis, whereas in Figure~\ref{fig:snr}-(b), the SNR is fixed and we vary the average expected degree on the $X$ axis. %
 The results show that VIPS consistently outperforms MFVI, indicating the advantage of the added dependence structure. Note that we plot BP with the model parameters initialized at true $(p_0,q_0)$ , since it is sensitive to initialization setting, and behaves poorly with mis-specified ones. Despite this, VIPS is largely comparable to BP and Spectral Clustering. For average degree 20 (Figure~\ref{fig:snr}-(b)), BP outperforms all other methods, because of the correct parameter setting. This NMI value becomes 0.4 with high variance, if we provide initial $\hat{p},\hat{q}$ values to match the average degree but $\hat{p}/\hat{q}=10$. In contrast, VIPS is much more robust to the initial choice of $\hat{p},\hat{q}$, which we show in Section~\ref{sec:gen} of the Appendix. 

\vspace{-3pt}
\begin{figure}[ht]
    \centering
\begin{tabular}{cc}
(a)&(b) \\ 
 \includegraphics[width=0.23\textwidth]{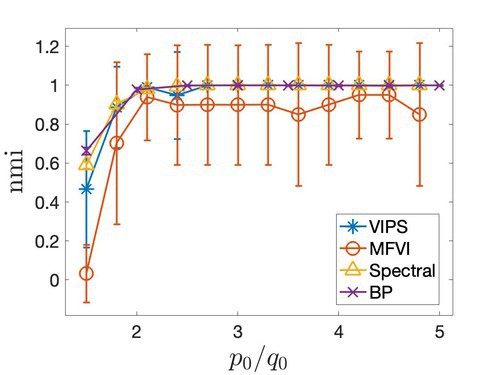}& 
\includegraphics[width=0.23\textwidth]{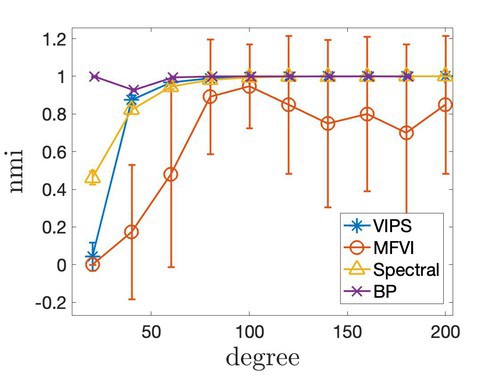}
\end{tabular}
  \caption { Comparison of NMI under different SNR $p_0/q_0$ and network degrees. The lines and error bars are means and standard deviations from 20 random trials. (a) Vary $p_0/q_0$ with degree fixed at 70.  (b) Vary the degree with $p_0/q_0 = 2$. In both figures $n=2000$. }
\label{fig:snr}
\end{figure}

Additional experiments (Appendix, section~\ref{sec:expaddl}) show that VIPS with fixed mis-specified parameters (within reasonable deviation from the truth), fixed true parameters and parameters updated with Eq.~\eqref{eq:hatpq} converge to the truth when initialized by independent Bernoulli's.

\section{Discussion and Generalizations }
\label{sec:discuss}
In this paper, we propose a simple Variational Inference algorithm with Pairwise Structure (VIPS) in a SBM with two equal sized communities. VI has been extensively applied in the latent variable models mainly due to their scalability and flexibility for incorporating changes in model structure. However, theoretical understanding of the convergence properties is limited and mostly restricted to the mean field setting with fully factorized variational distributions (MFVI). Theoretically we prove that in a SBM with two equal sized communities, VIPS can converge to the ground truth with probability tending to one for different random initialization schemes and a range of graph densities. In contrast, MFVI only converges for a constant fraction of Bernoulli(1/2) random initializations.  We consider settings where the model parameters are known, estimated or appropriately updated as part of the iterative algorithm. %

Though our main results are for $K=2,\pi=0.5$, we conclude with a discussion on generalizations to unbalanced clusters and SBMs with $K>2$ equal communities. 
To apply VIPS for general $K>2$ clusters, we will have $K^2-1$ categorical distribution parameters $\psi^{cd}$ for $c,d \in \{1,2,\ldots,K\}$ and marginal likelihood $\phi_1, \ldots, \phi_{K-1}$,  $\xi_1, \ldots, \xi_{K-1}$. The updates are similar to Eq.~\eqref{eq:thetaoo} and Eq.~\eqref{eq:phixi} and are deferred to the Appendix (section~\ref{sec:gen}). Similar to the $K=2$ case, we update the local and global parameters iteratively. As for the unbalanced case (see Appendix Section~\ref{sec:gen}), the updates involve an additional term which is the logit of $\pi$. We assume that $\pi$ is known and fixed. %

\begin{figure}[ht]
    \centering
\begin{tabular}{cc}
(a)&(b) \\ 
 \includegraphics[width=0.23\textwidth]{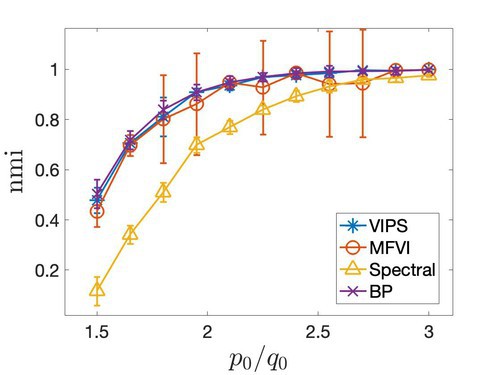}& 
\includegraphics[width=0.23\textwidth]{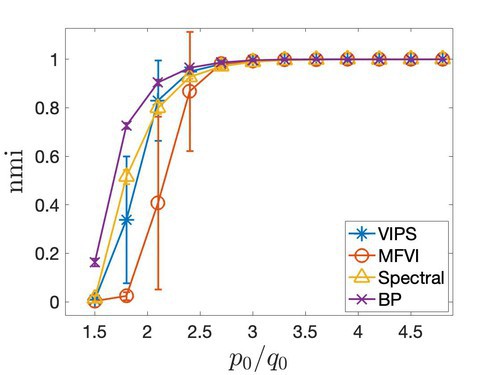} \\
\end{tabular}
  \caption { Comparison of VIPS, MFVI, Spectral and BP using error-bars from 20 random trials for $n=2000$, average degree 50, $p_0/q_0$ is changed on $X$ axis.  (a)~~ $\pi = 0.3$ (b) ~~ K =3, $B=(p-q)I+qJ$. For BP, MFVI and VIPS, we use true parameters. }
\label{fig:snr2}
\end{figure}
In Figure \ref{fig:snr2}-(a), we show results for unbalanced SBM with $\pi = 0.3$, which is assumed to be known. In Figure \ref{fig:snr2}-(b), similar to the setting in \citep{mukherjee2018mean}, we consider a SBM with three equal-sized communities. 
The parameters are set as $n = 2000$, average degree $50$, $p_0$ and $q_0$ are changed to get different SNR values and the random initialization is from $\text{Dirichlet}(1,1,1)$.  For a fair comparison of VIPS, MFVI and BP, we use the true $p_0,q_0$ values in all three algorithms;  robustness to parameter specification of VIPS is included in the Appendix~\ref{sec:gen}. We see that for the unbalanced setting (Figure~\ref{fig:snr2}-(a)) VIPS performs as well as BP and better than Spectral Clustering. For the $K=3$ setting (Figure~\ref{fig:snr2}-(b)) VIPS performs worse than BP and Spectral for very low SNR values, whereas for higher SNR it performs comparably to Spectral and BP, and better than MFVI, which has much higher variance. %

\newpage
\numberwithin{equation}{section}
\renewcommand\thefigure{A.\arabic{figure}} 
\renewcommand\thelemma{A.\arabic{lemma}}  
\allowdisplaybreaks
\onecolumn
\title{\vspace{-1cm}Supplementary material for ``A Theoretical Case Study of Structured Variational Inference for Community Detection''}
\date{}
\appendix
\maketitle
\allowdisplaybreaks
\vspace{-1cm}
\begin{quote}
This supplementary document contains detailed proofs and derivation of theoretical results presented in the main paper  ``A Theoretical Case Study of Structured Variational Inference for Community Detection'', and additional experimental results. In particular, Section~\ref{sec:elboderiv} contains the detailed derivation of updates of the Variational Inference with Pairwise Structure (VIPS) algorithm. Section~\ref{sec:proofmain} contains detailed proofs of the theoretical results presented in the main paper. Section~\ref{sec:gen} contains details on how to generalize VIPS to $K=2$ with unbalanced community sizes and $K=3$ with equal sized communities, and experimental results on robustness to parameter mis-specification. Section~\ref{sec:expaddl} contains additional experimental results and figures. 
\end{quote}

\section{Detailed Derivation of the Updates of VIPS}
\label{sec:elboderiv}
In the main paper \eqref{eq:elbo}, the Evidence Lower BOund (ELBO) for pairwise structured variational inference is 
\bas{ 
\cL(Q;\pi,B) =& \half \E_Q \sum_{i\neq j,a,b} Z_{ia}Z_{jb}(A_{ij}^{zz}\alpha_{ab}+f(\alpha_{ab}))+\half \E_Q \sum_{i\neq j,a,b} Y_{ia}Y_{jb}(A_{ij}^{yy}\alpha_{ab}+f(\alpha_{ab})) \\ \notag
+& \E_Q \sum_{i\neq j,a,b} Z_{ia}Y_{jb}(A_{ij}^{zy}\alpha_{ab}+f(\alpha_{ab}))+\E_Q \sum_{i,a,b} Z_{ia}Y_{ib}(A_{ii}^{zy}\alpha_{ab}+f(\alpha_{ab})) \\
-& \sum_{i=1}^m \kl(Q(z_i,y_i) || P(z_i)P(y_i))
}
where $\alpha_{ab} = \log(B_{ab}/(1-B_{ab}))$ and $f(\alpha)=-\log(1+e^{\alpha})$.  
Denote the first four terms in ELBO as $T_1$, $T_2$, $T_3$, $T_4$, where $T_1$, $T_2$ correspond to the likelihood of the blocks $A^{zz}$ and $A^{yy}$ in the adjacency matrix, $T_3$ corresponds to the likelihood of $(z_i,y_j), i \neq j$ and $T_4$ corresponds to $(z_i,y_i)$. Plugging in the marginal density of the independent nodes in $T_1$, $T_2$, $T_3$ and joint density of the dependent nodes in $T_4$, we have
\ba{
T_1 =& \half \sum_{i \neq j} \Big\{ [(1-\phi_i)(1-\phi_j)+\phi_i\phi_j](A_{ij}^{zz}\log\frac{p}{1-p} + \log(1-p))+\\  \notag
& [(1-\phi_i)\phi_j+\phi_i(1-\phi_j)](A_{ij}^{zz}\log\frac{q}{1-q} + \log(1-q))\Big\}
\\  
T_2 =& \half \sum_{i \neq j} \Big\{ [(1-\xi_i)(1-\xi_j)+\xi_i\xi_j](A_{ij}^{yy}\log\frac{p}{1-p} + \log(1-p))+\\ \notag
& [(1-\xi_i)\xi_j+\xi_i(1-\xi_j)](A_{ij}^{yy}\log\frac{q}{1-q} + \log(1-q))\Big\}
\\  
T_3 =& \sum_{i \neq j} \Big\{ [(1-\phi_i)(1-\xi_j)+\phi_i\xi_j](A_{ij}^{zy}\log\frac{p}{1-p} + \log(1-p))+\\ \notag
& [(1-\phi_i)\xi_j+\phi_i(1-\xi_j)](A_{ij}^{zy}\log\frac{q}{1-q} + \log(1-q))\Big\}
\\  
T_4 =& \sum_{i} \Big\{ (1-\psi_i^{01}-\psi_i^{10})(A_{ii}^{zy}\log\frac{p}{1-p} + \log(1-p))+\\  \notag
& (\psi_i^{01}+\psi_i^{10})(A_{ii}^{zy}\log\frac{q}{1-q} + \log(1-q))\Big\}
}
\\
The KL regularization term \eqref{eq:elbo-kl} is
\bas{ 
\kl(Q(z_i,y_i) || P(z_i)P(y_i)) =& \psi_i^{00}\log\frac{\psi_i^{00}}{(1-\pi)^2}+\psi_i^{01}\log\frac{\psi_i^{01}}{\pi(1-\pi)}+\\ 
&\psi_i^{10}\log\frac{\psi_i^{10}}{\pi(1-\pi)} + \psi_i^{11}\log\frac{\psi_i^{11}}{\pi^2} \\ 
=& \sum_{0\leq c,d \leq 1} \psi_i^{cd}\log\frac{\psi_i^{cd}}{\pi^c\pi^d(1-\pi)^{1-c}(1-\pi)^{1-d}}
}

To take the derivative of $\cL(Q;\pi,B)$ with respect to $\psi_i^{cd}, cd \neq 0$, we first have the derivative of the KL term
\ba{
\pde{\psi_i^{cd}} \kl(Q(z_i,y_i) || P(z_i)P(y_i)) =& \log \frac{\psi_i^{cd}}{\pi^{c+d}(1-\pi)^{2-c-d}} - \log \frac{\psi_i^{00}}{(1-\pi)^{2}} \\
=& \log \frac{\psi_i^{cd}}{1-\psi_i^{01}-\psi_i^{10}-\psi_i^{11}} ~~~~~~~(\pi = \half)
\label{eq:regu}
}
Denote the right hand side of Eq. \eqref{eq:regu} as  $\theta_i^{cd} \coloneqq \log \frac{\psi_i^{cd}}{1-\psi_i^{01}-\psi_i^{10}-\psi_i^{11}}$.  For the reconstruction terms, denoting $T(a,p) \coloneqq  a\log(\frac{p}{1-p})+\log(1-p)$ for simplicity, the derivative can be computed as
\ba{
\pde{\psi_i^{10}} (\sum T_k) =& \sum_{j,j\neq i} \Big[(2\phi_j - 1)T(A_{ij}^{zz},p)- (2\phi_j - 1)T(A_{ij}^{zz},q) \Big] + \\ \notag
& \sum_{j,j\neq i} \Big[(2\xi_j - 1)T(A_{ij}^{zy},p)- (2\xi_j - 1)T(A_{ij}^{zy},q) \Big] + \\ \notag
& \Big[-T(A_{ii}^{zy},p) + T(A_{ii}^{zy},q) \Big]
}

\ba{
\pde{\psi_i^{01}} (\sum T_k) =& \sum_{j,j\neq i} \Big[(2\xi_j - 1)T(A_{ij}^{yy},p)- (2\xi_j - 1)T(A_{ij}^{yy},q) \Big] + \\ \notag
& \sum_{j,j\neq i} \Big[(2\phi_j - 1)T(A_{ji}^{zy},p)- (2\phi_j - 1)T(A_{ji}^{zy},q) \Big] + \\ \notag
& \Big[-T(A_{ii}^{zy},p) + T(A_{ii}^{zy},q) \Big]
}

\ba{
\pde{\psi_i^{11}} (\sum T_k) =& \sum_{j,j\neq i} \Big[(2\phi_j - 1)T(A_{ij}^{zz},p)- (2\phi_j - 1)T(A_{ij}^{zz},q) \Big] + \\ \notag
&\sum_{j,j\neq i} \Big[(2\xi_j - 1)T(A_{ij}^{yy},p)- (2\xi_j - 1)T(A_{ij}^{yy},q) \Big] + \\ \notag
& \sum_{j,j\neq i} \Big[(2\xi_j - 1)T(A_{ij}^{zy},p)- (2\xi_j - 1)T(A_{ij}^{zy},q) \Big] + \\ \notag
& \sum_{j,j\neq i} \Big[(2\phi_j - 1)T(A_{ji}^{zy},p)- (2\phi_j - 1)T(A_{ji}^{zy},q) \Big] 
}

Setting the derivatives to $0$ we get the update for $\theta$ as \eqref{eq:thetazo}, \eqref{eq:thetaoz}, \eqref{eq:thetaoo}.

\section{Proofs of Main Results}
\label{sec:proofmain}
To prove Theroem~\ref{thm: convergence}, we first need a few lemmas. First we have the following lemma for the parameters $p$, $q$ and $\lambda$. 

\begin{lemma}
	If $p\asymp q \asymp \rho_n, \rho_n \to 0$ and $p - q = \Omega(\rho_n)$, then
	\begin{align}
	\lambda-q&=\Omega(\rho_n) > 0, \label{eq:lambdalb_sep} \\
	\frac{p+q}{2}-\lambda&=\Omega(\rho_n)>0. \label{eq:lambdaub_sep}
	\end{align}
	\label{prop:pq}
\end{lemma}

\begin{proof}
	The proof follows from Proposition 2 in \cite{sarkar2019random}.
\end{proof}

In the proof, we utilize the spectral property of the population matrix $P$ and generalize it to the finite sample case by bounding the term related to the residual $R = A - P$. We use Berry-Esseen Theorem to bound the residual terms conditioning on $u$. 

\begin{lemma}[Berry-Esseen bound] 
	Define 
	\ba{
	r_i = \sum_{j=1}^n (A_{ij} - P_{ij})(u(j) - \half),
	}
	where $u$ and $A$ are independent. 
	\label{lem:be_bound}
	\begin{align*}
	\sup_{x\in\mathbb{R}}|P\left(r_i / \sigma_u \leq x \mid u \right) - \Phi(x)| \leq \frac{C\rho_u}{\sigma^3_u},
	\end{align*}
	where $C$ is a general constant, $\Phi(\cdot)$ is the CDF of standard Gaussian, $\rho_u$ and $\sigma_u$ depend on $u$. 
	\label{lem:berry-e}
\end{lemma}

\begin{proof}
	Since $r_i$ is the sum of independent, mean zero random variables, the sum of the conditional variances is
	\bas{
		\sigma_u^2 = \text{Var}(r_i | u) = &p(1-p)\sum_{i \in G_1}(u(i) - \half)^2 + q(1-q) \sum_{i \in G_2}(u(i) - \half)^2, 
	}
	and the sum of the conditional absolute third central moments is
	\bas{
		\rho_u = &p(1-p)(1-2p+2p^2)\sum_{i \in G_1}|u(i) - \half |^3 + q(1-q)(1-2q+2q^2) \sum_{i \in G_2}|u(i)  - \half |^3.
	}
	
	The desired bound follows from the Berry-Esseen Theorem.
\end{proof} 
The next lemma shows despite the fact that $A$ introduces some dependency among $r_i$ due to its symmetry, we can still treat $r_i$ as almost iid. 
\begin{lemma}[McDiarmid's Inequality]
	\label{lem:mcd_bound}
	Let $r_i$ be the noise defined in Lemma~\ref{lem:berry-e} and let $h(r_i)$ be a bounded function with $\|h \|_{\infty} \leq M$. Then 
	$$P\left( \left| \frac{2}{n}\sum_{i\in \mathcal{A}} h(r_i) - \E(h(r_i) | u) \right| > w \mid u \right) \leq \exp\left(-\frac{c_0 w^2}{nM} \right)$$
	for some general constant $c_0$, provided $|\mathcal{A}|=\Theta_P(n)$. 
\end{lemma}

\begin{proof}
	The proof follows from Lemma 20 in \cite{sarkar2019random}.
\end{proof}

\begin{lemma}
Let $r_i$ be defined as in Lemma~\ref{lem:berry-e} and assume $A$ and $u$ are independent, we have  $\sup_{i\in\mathcal{A}} |r_i| = O_P(\sqrt{n\rho_n \log n})$ if the index set $|\mathcal{A}|=\Theta_P(n)$. 
\label{lem:uniform_r}
\end{lemma}
\begin{proof}
Since $r_i$ is the sum of independent bounded random variables, for all i, $r_i = O_P(\sqrt{n\rho_n})$. By Hoeffding inequility, we know for all $t>0$
\bas{
P(|r_i| > t) \leq \exp(-\frac{t^2}{2n\rho_n})
}
and by the union bound 
\bas{
P(\sup_i |r_i| > t) \leq \exp(C\log n -\frac{t^2}{2n\rho_n})
}
For $\forall \epsilon>0$, let $t = C_\epsilon \sqrt{n\rho_n \log n}$ with $n^{\frac{C_\epsilon^2}{2}-1} > 1/\epsilon$, then by definition $\sup_i |r_i|  = O_P(\sqrt{n\rho_n \log n})$
\end{proof}

Next we have a lemma ensuring the signal in the first iteration is not too small.
\begin{lemma}[Littlewood-Offord]
	\label{lem:anti_conc}
	Let $s_1 = (p-\lambda)\sum_{i\in G_1} (u^{(0)}(i)-1/2)+(q-\lambda)\sum_{i\in G_2} (u^{(0)}(i)-1/2)$, $s_2= (q-\lambda)\sum_{i\in G_1} (u^{(0)}(i)-1/2)+(p-\lambda)\sum_{i\in G_2} (u^{(0)}(i)-1/2)$.
	Then
	$$P\left( |s_1| \leq c \right) \leq B\cdot \frac{c}{\rho_n\sqrt{n}}$$
	for $c>0$ and $B$ as constant. The same bound holds for $ |s_2|, |s_1-s_2|$.
\end{lemma}
\begin{proof}
	Noting that $2u^{(0)}(i)-1\in \{-1,1\}$ each with probability $1/2$, and Lemma~\ref{prop:pq}, this is a direct consequence of the Littlewood-Offord bound in~\cite{erdos1945lemma}.
\end{proof}

Finally, we have the following upper and lower bound for some general update $\phi_i$.
\begin{lemma}
	Assume $\phi_i$ has the update form $\phi_i= (a+e^{4t(s+r_i)}) / (b+e^{4t(s+r_i)})$ for $i \in [m]$, $b > a > 0$ and $b-a$, $(b-a)/b$ are of constant order. $r_i$ is defined as in Lemma \ref{lem:berry-e}. Let set $\mathcal{A} \subset [m]$, with $\Delta > 0$, we have
	\bas{
		\sum_{i \in \mathcal{A}} \phi_i \geq & |\mathcal{A}|  - \frac{b-a}{b}  |\mathcal{A}|  \Phi(\frac{-s + \Delta}{\sigma_u})- C' |\mathcal{A}|  \frac{\rho_u}{\sigma_u^3} - C'' |\mathcal{A}| e^{-4t \Delta} - O_P(\sqrt{ |\mathcal{A}| }), \\
		\sum_{i \in \mathcal{A}} \phi_i \leq &  |\mathcal{A}|  - \frac{b-a}{b}  |\mathcal{A}|  \Phi(\frac{-s - \Delta}{\sigma_u}) + C' |\mathcal{A}|  \frac{\rho_u}{\sigma_u^3}+  |\mathcal{A}| e^{-4t \Delta} + O_P(\sqrt{ |\mathcal{A}| }).
	}
	\label{lemma:bounds}
\end{lemma}

\begin{proof}
	Define the set $J^+ =  \{ i : r_i > - s + \Delta\}$, $\Delta \geq 0$. For $i \in \mathcal{A} \cap J^+$
	\bas{
		\phi_i= \frac {a+e^{4t(s+r_i)}} {b+e^{4t(s+r_i)}} \geq  \frac {a+e^{4t\Delta}} {b+e^{4t\Delta}} \geq 1 - (b-a) e^{-4t\Delta} 
	}
	For $i \in (\mathcal{A} \cap J^+)^c$, $\phi_i \geq  a/b$,  therefore
	\bas{
		\sum_{i \in \mathcal{A}} \phi_i \geq & |\mathcal{A} \cap J^+| (1 - (b-a)e^{-4t\Delta}) + \frac{a}{b} (|\mathcal{A}| -  |\mathcal{A}\cap J^+| ) \\
		=&  |\mathcal{A} \cap J^+| (\frac{b-a}{b} - (b-a)e^{-4t\Delta}) + \frac{a}{b} |\mathcal{A}| 
	}
	
	By Lemmas \ref{lem:berry-e} and \ref{lem:mcd_bound}, we have
	\bas{
		|\mathcal{A} \cap J^+|  =& \sum_{i \in \mathcal{A}} \mathbf{1} [r_i > -s + \Delta]\\
		=&  |\mathcal{A} | \cdot P(r_i > -s + \Delta) + O_P(\sqrt{ |\mathcal{A}| }) \\
		\geq &  |\mathcal{A} | \cdot (1 - \Phi(\frac{-s + \Delta}{\sigma_u}) - C_0 \frac{\rho_u}{\sigma_u^3})  +  O_P(\sqrt{|\mathcal{A}|}).
	} 
	Combining the above,
	\bas{
		\sum_{i \in \mathcal{A}} \phi_i \geq & |\mathcal{A}|  - \frac{b-a}{b}  |\mathcal{A}|  \Phi(\frac{-s + \Delta}{\sigma_u})- C' |\mathcal{A}|  \frac{\rho_u}{\sigma_u^3} - C'' |\mathcal{A}| e^{-4t \Delta} - O_P(\sqrt{ |\mathcal{A}| })
	}
	
	Similarly, define the set $J^- = \{i : r_i < - s- \Delta\}$, $\Delta \geq 0$.  For $i \in \mathcal{A} \cap J^-$,
	\bas{
		\phi_i= \frac {a+e^{4t(s+r_i)}} {b+e^{4t(s+r_i)}} \leq  \frac {a+e^{-4t\Delta}} {b+e^{-4t\Delta}} \leq \frac{a}{b} +  e^{-4t\Delta} 
	}
	For $i \in (\mathcal{A} \cap J^-)^c$, $\phi_i \leq  1$, so
	\bas{
		\sum_{i \in  \mathcal{A}}  \phi_i \leq &  | \mathcal{A} \cap J^-|(\frac{a}{b} + e^{-4t\Delta}) + (| \mathcal{A}| -  | \mathcal{A} \cap J^-| )  \\
		=& | \mathcal{A}| -  | \mathcal{A} \cap J^-| (1-\frac{a}{b} - e^{-4t\Delta}) + O_P(\sqrt{ |\mathcal{A}| })
	}
	By Lemmas \ref{lem:berry-e} and \ref{lem:mcd_bound},
	\bas{
		| \mathcal{A} \cap J^-|  \geq  | \mathcal{A}|\cdot (\Phi(\frac{-s - \Delta}{\sigma_u}) - C_0 \frac{\rho_u}{\sigma_u^3})-  O_P(\sqrt{ |\mathcal{A}| })
	}
	so
	\bas{
		\sum_{i \in \mathcal{A}} \phi_i \leq &  |\mathcal{A}|  - \frac{b-a}{b}  |\mathcal{A}|  \Phi(\frac{-s - \Delta}{\sigma_u}) + C' |\mathcal{A}|  \frac{\rho_u}{\sigma_u^3}+  |\mathcal{A}| e^{-4t \Delta} + O_P(\sqrt{ |\mathcal{A}| })
	}
\end{proof}

\begin{proof}[\textbf{Proof of Theorem \ref{thm: convergence}}]  \hfill 
	Throughout the proof, we assume $A$ has self-loops for convenience, which does not affect the asymptotic results. 
	
	\textbf{Analysis of the first iteration in the first meta iteration:}
	
	For random initialized $u^{(0)}$, the initial signal $|\ip{u^{(0)}}{v_{2}}|=O_P(\sqrt{n})$.  Using the graph split $A^{(0)}$, we write the update of $\theta^{10}$ as
	\ba{
		\theta^{10} =& 4t([6(A^{(0)})^{zz}, 6(A^{(0)})^{zy}] - \lambda J)(u^{(0)} - \half \one_n) \notag\\
		=&\underbrace{ 4t([P^{zz}, P^{zy}] - \lambda J)(u^{(0)} - \half\one_n)}_\text{signal}+ \underbrace{4t [6(A^{(0)})^{zz}-P^{zz}, 6(A^{(0)})^{zy}-P^{zy}](u^{(0)}- \half\one_n), }_{\text{noise }}
		\label{eq:decomp_theta10}
	}
	where $P$ is the population matrix of $A$. Denote $R^{(0)}=6A^{(0)}-P$ and
	$r^{(0)} = [(R^{(0)})^{zz}, (R^{(0)})^{zy}](u^{(0)}- \half\one)$
	Since $P$ has singular value decomposition as $P =  \frac{p+q}{2} \mathbf{1}_n \mathbf{1}_n^T + \frac{p-q}{2} v_2v_2^T $, the signal part is blockwise constant and we can write 
	\ba{
		\theta^{10} = 4t(s_1 \cone + s_2 \ctwo +  r^{(0)}),
		\label{eq:theta10}
	}
	where
	\ba{
		s_1 = &  ( \frac{p+q}{2} -\lambda)(\ip{u^{(0)}}{\mathbf{1}_n}-m)  + (  \frac{p-q}{2}) \ip{u^{(0)}}{v_2} \nonumber\\
		s_2 = & ( \frac{p+q}{2} -\lambda) (\ip{u^{(0)}}{\mathbf{1}_n}-m)  - (  \frac{p-q}{2}) \ip{u^{(0)}}{v_2} 
		\label{eq:s12}
	}
	By \eqref{eq:phixi}, since we initialize with $\theta^{01}, \theta^{11}=0$, the marginal probabilities are updated as
	\ba{
		\phi^{(1)}_1 = \frac{1+e^{\theta^{10}}}{3+e^{\theta^{10}}}, \quad \xi^{(1)}_1 = \frac{2}{3+e^{\theta^{10}}}
		\label{eq:iter0}
	}
	Next we show the signal $|\ip{u}{v_{2}}|$ increases from $O_P(\sqrt{n})$ to $\Omega_P(n\sqrt{\rho_n})$. (We omit the superscript on logits $s$,$x$ and $y$ now for simplicity.) Since 
	\bas{
		\ip{u_1^{(1)}}{v_{2}} = \ip{\phi^{(1)}_{1i}}{v_{21}}  + \ip{\xi^{(1)}_{1i}}{v_{22}}  = \sum_{i \in C_1} \phi_i^{(1)} - \sum_{i \in C_2} \phi_i^{(1)}  + \ip{\xi^{(1)}}{v_{22}}
	}
	we use Lemma \ref{lemma:bounds} to bound $\sum_{i \in C_1} \phi_i^{(1)} $ and $\sum_{i \in C_2} \phi_i^{(1)} $. Since $s_1$ and $s_2$ depends on $u^{(0)}$, we consider two cases conditioning on $u^{(0)}$.

	\textit{Case 1}: $s_1 > s_2$. By Lemma \ref{lemma:bounds}, let $\Delta = \frac{1}{4}(s_1 - s_2)$ with  $\mathcal{A} = C_1, C_2$, $(a,b) = (1,3)$, conditioning on $u^{(0)}$,
	\bas{
		&\sum_{i \in C_1} \phi^{(1)}_{1i} \geq \frac{n}{6}(1 - \Phi(- \frac{s_1 - \frac{1}{4}(s_1 - s_2)}{\sigma_u}) ) + \frac{n}{12} - C' n  \frac{\rho_u}{\sigma_u^3} - C'' n e^{-t(s_1-s_2) } - O_P(\sqrt{n}), \\
		&\sum_{i \in C_2} \phi^{(1)}_{1i} \leq \frac{n}{4} - \frac{n}{6}\Phi(- \frac{s_2 + \frac{1}{4}(s_1 - s_2)}{\sigma_u}) + C' n  \frac{\rho_u}{\sigma_u^3} + C'' n e^{-t(s_1-s_2) } + O_P(\sqrt{n}),
	} 
	where the $O_P(\sqrt{n})$ term can be made uniform in $u^{(0)}$.
	So we have
	\ba{
		\ip{\phi^{(1)}_{1} }{v_{21}} \geq & \frac{n}{6} (\Phi(- \frac{s_2 + \frac{1}{4}(s_1 - s_2)}{\sigma_u}) - \Phi(- \frac{s_1 - \frac{1}{4}(s_1 - s_2)}{\sigma_u}))  \notag\\
		& \qquad \qquad -  C' n  \frac{\rho_u}{\sigma_u^3} - C'' n e^{-t(s_1-s_2) } - O_P(\sqrt{n}) 	\notag\\
		\geq & \frac{n}{6\sqrt{2\pi}}(\frac{s_1 - s_2}{2\sigma_u}) \exp\big({-\frac{s_1^2 \lor s_2^2}{2\sigma_u^2}}\big) 	\notag\\
		& \qquad \qquad - C' n  \frac{\rho_u}{\sigma_u^3} - C'' n e^{-t(s_1-s_2) } - O_P(\sqrt{n}).
		\label{eq:iter1_finalbound}
	}
	Here to approximate the CDF $\Phi$, we have used 
	\begin{align}
	|\Phi(x)-1/2| & = \frac{1}{\sqrt{2\pi}} \int_0^{|x|} e^{-u^2/2}du 	\notag\\ 
	&  \geq \frac{|x|}{\sqrt{2\pi}}e^{-x^2/2}.
	\label{eq:normal_cdf_lb}
	\end{align}
	
	\textit{Case 2}: $s_1 < s_2$.  The same analysis applies with $s_1$ and  $s_2$ interchanged.  
	
	Combining \textit{Case 1} and \textit{Case 2}, for any given $u^{(0)}$, 
	\ba{
		|\ip{\phi^{(1)}_1 }{v_{21}} | & \geq \frac{n}{6\sqrt{2\pi}}(\frac{|s_1 - s_2|}{2\sigma_u}) \exp\big({-\frac{s_1^2 \lor s_2^2}{2\sigma_u^2}}\big) 	\notag\\
		& \qquad \qquad - C' n  \frac{\rho_u}{\sigma_u^3} - C'' n e^{-t|s_1-s_2| } - O_P(\sqrt{n}).
	}
	We note that $|s_1|$, $|s_2|$, $|s_1-s_2|$ are of order $\Omega_P(\sqrt{n}\rho_n)$ by Lemma~\ref{lem:anti_conc}. Also $\sigma^2_u, \rho_u \asymp n\rho_n$,  $e^{-4t|s_1-s_2|}  = \exp(-\Omega(\rho_n \sqrt{n}))$. We can conclude that $ |\ip{\phi^{(1)}_1 }{v_{21}} |  = \Omega_P(n \sqrt{\rho_n})$.
	
	For $\ip{\xi^{(1)}_1}{v_{22}}$ we have
	
	\bas{
		|\ip{\xi^{(1)}_1}{v_{22}}| =& \Big| \sum_{i \in C'_1 } \xi_i^{(1)} - \sum_{i \in C'_2 } \xi_i^{(1)}\Big|  =  \Big|\sum_{i \in C'_2 } \phi_i^{(1)} - \sum_{i \in C'_1 } \phi_i^{(1)} + |C'_1| -|C'_2| \Big|\\
		=& O_P(\sqrt{n})
	}
	Therefore we have $|\ip{u_1^{(1)}}{v_{2}} |= \Omega_P(n \sqrt{\rho_n})$. By \eqref{eq:iter0}, $\ip{u_1^{(1)}}{\one} - m= 0$. %
	
	Due to the symmetry in $s_1$ and $s_2$, WLOG in the following analysis, we assume $\ip{u_1^{(1)}}{v_{2}} > 0$ (equivalently $s_1>s_2$). 
	
	\textbf{Analysis of the second iteration in the first meta iteration:}
	
	Similar to~\eqref{eq:decomp_theta10}, we can write 
	\bas{
		\theta^{01} =& 4t([6(A^{(1)})^{yz}, 6(A^{(1)})^{yy}] - \lambda J)(u_1^{(1)} - \half\one_n) \\
		=&\underbrace{ 4t([P^{yz}, P^{yy}] - \lambda J)(u_1^{(1)} - \half\one_n)}_\text{signal}+ \underbrace{4t (R^{(1)})^{yz}(\phi^{(1)}_1 - \half\one_m) + 4t (R^{(1)})^{yy}(\xi^{(1)}_1 - \half\one_m)}_{\text{noise $\coloneqq 4t r_i^{(1)}$}}.
	}
	Noting the signal part is blockwise constant, we have 
	\bas{
		\theta^{01} = 4t(x_1 \cpone + x_2 \cptwo +  r^{(1)}),
	}
	where
	\bas{
		x_1 = &  ( \frac{p+q}{2} -\lambda)(\ip{u_1^{(1)}}{\mathbf{1}_n}-m)  + (  \frac{p-q}{2}) \ip{u_1^{(1)}}{v_2} \\
		x_2 = & ( \frac{p+q}{2} -\lambda) (\ip{u_1^{(1)}}{\mathbf{1}_n}-m)  - (  \frac{p-q}{2}) \ip{u_1^{(1)}}{v_2}  \\
	}
	By \eqref{eq:iter0}, $\ip{u_1^{(1)} }{\one_n} - m = 0$ and we have
	\bas{
		x_1 = & (  \frac{p-q}{2}) \ip{u_1^{(1)}}{v_2},  \\
		x_2 = & - x_1. 
	}
	It follows then from the first iteration that $x_1,-x_2  = \Omega_P(n \rho_n^{3/2})$. The update for $u_2^{(1)}$ is
	\ba{
		\phi^{(1)}_2 = \frac{1+e^{\theta^{10}}}{2+ e^{\theta^{10}} + e^{\theta^{01}}}, \quad \xi^{(1)}_2 = \frac{1+e^{\theta^{01}}}{2+ e^{\theta^{10}} + e^{\theta^{01}}}
		\label{eq:iter1}
	}
	Since the signal part of $\theta^{10}$ and $\theta^{01}$ are blockwise constant on $C_1$, $C_2$ and $C'_1$, $C'_2$ respectively, $\ip{u_2^{(1)}}{v_{2}}$ can be calculated as
	\bas{
		\ip{\phi^{(1)}_2}{v_{21}} =& \sum_{i \in C_{11}}\frac{1+e^{4t(s_1+ r_i^{(0)})}}{2+e^{4t(s_1+ r_i^{(0)})}+e^{4t(x_1+r_i^{(1)})}}  + \sum_{i \in C_{12}}\frac{1+e^{4t(s_1+ r_i^{(0)})}}{2+e^{4t(s_1+r_i^{(0)})}+e^{4t(x_2+r_i^{(1)})}}  \\
		& - \sum_{i \in C_{21}} \frac{1+e^{4t(s_2+r_i^{(0)})}}{2+e^{4t(s_2+r_i^{(0)})}+e^{4t(x_1+r_i^{(1)})}} - \sum_{i \in C_{22}} \frac{1+e^{4t(s_2+r_i^{(0)})}}{2+e^{4t(s_2+ r_i^{(0)})}+e^{4t(x_2+r_i^{(1)})}}
	}

	\bas{
		\ip{\xi^{(1)}_2}{v_{22}} =& \sum_{i \in C_{11}}\frac{1+e^{4t(x_1+r_i^{(1)})}}{2+e^{4t(s_1+ r_i^{(0)})}+e^{4t(x_1+r_i^{(1)})}}  + \sum_{i \in C_{21}}\frac{1+e^{4t(x_1+r_i^{(1)})}}{2+e^{4t(s_2+r_i^{(0)})}+e^{4t(x_1+r_i^{(1)})}}  \\
		& - \sum_{i \in C_{12}} \frac{1+e^{4t(x_2+ r_i^{(1)})}}{2+e^{4t(s_1+r_i^{(0)})}+e^{4t(x_2+ r_i^{(1)})}} - \sum_{i \in C_{22}} \frac{1+e^{4t(x_2+ r_i^{(1)})}}{2+e^{4t(s_2+r_i^{(0)})}+e^{4t(x_2+r_i^{(1)})}}
	}

	In the case of $\ip{u_1^{(1)}}{v_2} > 0$, we know , $s_1 > s_2$ and $x_1 > 0 > x_2$. 
	We first show that $\ip{\phi^{(1)}_2}{v_{21}}$ is positive by finding a lower bound for the summations over $C_{12}, C_{21}, C_{22}$ (since the sum over $C_{11}$ is always positive).
	
	For the summation over $C_{12}$, note that $|x_2|$ dominates both $s_1$ and $r_i^{(0)}$, $r_i^{(1)}$ by Lemma \ref{lem:uniform_r}, we have 
	\bas{
		\sum_{i \in C_{12}}\frac{1+e^{4t(s_1+ r_i^{(0)})}}{2+e^{4t(s_1+r_i^{(0)})}+e^{4t(x_2+r_i^{(1)})}}  = \sum_{i \in C_{12}}\frac{1+e^{4t(s_1+ r_i^{(0)})}}{2+e^{4t(s_1+r_i^{(0)})}} +  n\exp(-\Omega_P(n\rho_n^{3/2})). 
	}
	To lower bound the first term, we use Lemma \ref{lemma:bounds} by first conditioning on $u^{(0)}$,
	\ba{
		& \sum_{i \in C_{12}}\frac{1+e^{4t(s_1+ r_i^{(0)})}}{2+e^{4t(s_1+r_i^{(0)})}} 	\notag\\
		\geq & \frac{n}{8}\left(1 - \frac{1}{2}\Phi(\frac{-s_1 + \Delta}{\sigma_u}) \right)- C' n \frac{\rho_u}{\sigma_u^3} - C''ne^{-4t\Delta}  - O_P(\sqrt{n}) 
		\label{eq:c12}
	}

	For the summation over $C_{22}$, 
	\ba{
		&\sum_{i \in C_{22}} \frac{1+e^{4t(s_2+r_i^{(0)})}}{2+e^{4t(s_2+ r_i^{(0)})}+e^{4t(x_2+r_i^{(1)})}} \leq \sum_{i \in C_{22}} \frac{1+e^{4t(s_2+r_i^{(0)})}}{2+e^{4t(s_2+ r_i^{(0)})}} 	\notag\\
		\leq & \frac{n}{8}\left( 1 - \frac{1}{2}  \Phi(\frac{-s_2 - \Delta}{\sigma_u}) \right) + C' n \frac{\rho_u}{\sigma_u^3} + C''ne^{-4t\Delta} + O_P(\sqrt{n})
		\label{eq:c22}
	}
	
	For the summation over $C_{21}$, $x_1$ dominates $s_2$ and $r_i^{(0)}, r_i^{(1)}$ by Lemma \ref{lem:uniform_r},
	\ba{
		\sum_{i \in C_{21}} \frac{1+e^{4t(s_2+r_i^{(0)})}}{2+e^{4t(s_2+r_i^{(0)})}+e^{4t(x_1+r_i^{(1)})}} = n\exp(-\Omega_P(n\rho_n^{3/2})).
		\label{eq:c21}
	}
	
	Combining \eqref{eq:c12} - \eqref{eq:c21}, setting $\Delta = \frac{1}{4}(s_1 - s_2)$, we have
	\bas{
		\ip{\phi^{(1)}_2}{v_{21}}
		& \geq  \frac{n}{8}\big[\frac{1}{2}  \Phi(\frac{-s_2 - \Delta}{\sigma_u}) - \frac{1}{2}\Phi(\frac{-s_1 + \Delta}{\sigma_u}) \big] 	- C' n \frac{\rho_u}{\sigma_u^3} - C''ne^{-4t\Delta} - O_P(\sqrt{n}) 	\\
		\geq  &  \frac{n}{16} \frac{1}{\sqrt{2\pi}}(\frac{s_1 - s_2}{\sigma_u} )  \exp\big({-\frac{s_1^2 \lor s_2^2}{2\sigma_u^2}}\big) -  C' n \frac{\rho_u}{\sigma_u^3} - C''ne^{-t(s_1-s_2)} - O_P(\sqrt{n}) 
	}
	by the same argument as \eqref{eq:iter1_finalbound}. As before, we can see that 
	\bas{
		\ip{\phi^{(1)}_2}{v_{21}} = \Omega_P(n \sqrt{\rho_n})
	}
	
	For $\ip{\xi^{(1)}_2}{v_{22}}$,  since $(1+e^x)/(2+e^x)\leq 1/2 + e^x$, we have
	\ba{
		&\sum_{i \in C_{12}} \frac{1+e^{4t(x_2+ r_i^{(1)})}}{2+e^{4t(s_1+r_i^{(0)})}+e^{4t(x_2+ r_i^{(1)})}} + \sum_{i \in C_{22}} \frac{1+e^{4t(x_2+ r_i^{(1)})}}{2+e^{4t(s_2+r_i^{(0)})}+e^{4t(x_2+r_i^{(1)})}} 	\notag\\
		\leq & \frac{n}{8} + \sum_{i \in C'_{2}} e^{4t(x_2 +r_i^{(1)})} + O_P(\sqrt{n}) 	\notag\\
		\leq & \frac{n}{8} + O_P(\sqrt{n}). 
		\label{eq:c12c22_xi}
	}

	For the other two sums, we have 
	\ba{
		\sum_{i \in C_{11}}\frac{1+e^{4t(x_1+ r_i^{(1)})}}{2+e^{4t(s_1+r_i^{(0)})}+e^{4t(x_1+ r_i^{(1)})}} & \geq \frac{n}{8} - O_{P}(\sqrt{n}) - n\exp(-\Omega_P(n\rho_n^{3/2})),	\notag\\
		& \geq \frac{n}{8}- O_P(\sqrt{n})
		\label{eq:c11_xi}
	}
	and 
	\ba{
		\sum_{i \in C_{21}}\frac{1+e^{4t(x_1+r_i^{(1)})}}{2+e^{4t(s_2+r_i^{(0)})}+e^{4t(x_1+r_i^{(1)})}}  \geq \frac{n}{8} - O_P(\sqrt{n})
		\label{eq:c21_xi}
	}
	
	Equations \eqref{eq:c12c22_xi} - \eqref{eq:c21_xi} imply
	\bas{
		\ip{\xi^{(1)}_2}{v_{22}} \geq \frac{n}{8} - O_P(\sqrt{n}).
	}
	
	Therefore $\ip{u_2^{(1)}}{v_2} \geq n/8  - O_P(\sqrt{n})$. Since by \eqref{eq:iter1},  $\phi^{(1)}_2 = \one_m - \xi^{(1)}_2$,  the inner product  $\ip{u_2^{(1)}}{\one} - m = 0$.

	\textbf{Analysis of the third iteration in the first meta iteration:}
	
	Similar to the previous two iterations, we can write 
	\bas{
		\theta^{11} = 4t(y_1 \cone + y_2 \ctwo +y_1 \cpone + y_2 \cptwo +  r^{(2)}),
	}
	where
	\bas{
		y_1 =   (  \frac{p-q}{2}) \ip{u_2^{(1)}}{v_2},  \quad 
		y_2 =   - y_1 
	}
	\bas{
		r^{(2)} = [(R^{(2)})^{zz}, (R^{(2)})^{zy}] (u_2-\half\one_n) + [(R^{(2)})^{yz}, (R^{(2)})^{yy}](u_2^{(1)}-\half\one_n).
	}
	It follows from the second iteration that $y_1, -y_2 = \Omega_P(n\rho_n)$. 
	
	The update for $u_3^{(1)}$ is
	\ba{
		\phi^{(1)}_3 = \frac{e^{\theta^{11}}+e^{\theta^{10}}}{1+ e^{\theta^{10}} + e^{\theta^{01}}+ e^{\theta^{11}}}, \quad \xi^{(1)}_3 = \frac{e^{\theta^{11}}+e^{\theta^{01}}}{1+ e^{\theta^{10}} + e^{\theta^{01}}+ e^{\theta^{11}}}
		\label{eq:iter2}
	}
	The  $\ip{u_3^{(1)}}{v_{2}}$ can be calculated as
	
	\ba{
		& \displaystyle\hspace*{-10pt}\ip{u_3^{(1)}}{v_{2}} =  \notag  \\
		&\displaystyle \hspace*{-10pt} \sum_{i \in C_{11}}\frac{2e^{8t(y_1+ r_i^{(2)})}+e^{4t(x_1+ r_i^{(1)})}+e^{4t(s_1+ r_i^{(0)})}} {1+e^{4t(s_1+ r_i^{(0)})}+e^{4t(x_1+r_i^{(1)})} + e^{8t(y_1+ r_i^{(2)})}}  + \sum_{i \in C_{12}}\frac{e^{4t(s_1+ r_i^{(0)})} - e^{4t(x_2+ r_i^{(1)})}}{1 + e^{4t r_i^{(2)}}+e^{4t(s_1+r_i^{(0)})}+e^{4t(x_2+r_i^{(1)})}} \notag \\
		&\hspace*{-10pt}  +  \displaystyle \sum_{i \in C_{21}} \frac{e^{4t(x_1+ r_i^{(1)})}-e^{4t(s_2+r_i^{(0)})}}{1+ e^{4t r_i^{(2)}} +e^{4t(s_2+r_i^{(0)})}+e^{4t(x_1+r_i^{(1)})}} - \sum_{i \in C_{22}} \frac{2e^{8t(y_2+ r_i^{(2)})}+e^{4t(s_2+r_i^{(0)})} + e^{4t(x_2+ r_i^{(1)})}}{1+e^{4t(s_2+ r_i^{(0)})}+e^{4t(x_2+r_i^{(1)})} + e^{8t(y_2+ r_i^{(2)})}}
		\label{eq:uv2}
	}
	
	Using the order of the $x$ terms and $y$ terms and Lemma \ref{lem:uniform_r}, we can lower bound $\ip{u_3^{(1)}}{v_{2}}$ by
	\ba{
		\ip{u_3^{(1)}}{v_{2}} & \geq \frac{n}{4} +  \sum_{i \in C_{12}}\frac{e^{4t(s_1+ r_i^{(0)})} }{1 + e^{4t r_i^{(2)}} +e^{4t(s_1+r_i^{(0)})}}  + \frac{n}{8}  - \sum_{i \in C_{22}} \frac{e^{4t(s_2+r_i^{(0)})} }{1+e^{4t(s_2+ r_i^{(0)})} } - O_P(\sqrt{n})	\notag\\
		&  \geq \frac{n}{4}- O_P(\sqrt{n}).
		\label{eq:u3_signal_lb}
	}

	Next we bound $\ip{u_3^{(1)}}{\one_n}- m$. 
		\ba{
		& \displaystyle\hspace*{-10pt}\ip{u_3^{(1)}}{\one_n} =  \notag  \\
		&\displaystyle \hspace*{-10pt} \sum_{i \in C_{11}}\frac{2e^{8t(y_1+ r_i^{(2)})}+e^{4t(x_1+ r_i^{(1)})}+e^{4t(s_1+ r_i^{(0)})}} {1+e^{4t(s_1+ r_i^{(0)})}+e^{4t(x_1+r_i^{(1)})} + e^{8t(y_1+ r_i^{(2)})}}  + \sum_{i \in C_{12}}\frac{2e^{4tr_i^{(2)}}+e^{4t(s_1+ r_i^{(0)})} + e^{4t(x_2+ r_i^{(1)})}}{1 + e^{4t r_i^{(2)}}+e^{4t(s_1+r_i^{(0)})}+e^{4t(x_2+r_i^{(1)})}} \notag \\
		&\hspace*{-10pt}  +  \displaystyle \sum_{i \in C_{21}} \frac{2e^{4tr_i^{(2)}}+e^{4t(x_1+ r_i^{(1)})}+e^{4t(s_2+r_i^{(0)})}}{1+ e^{4t r_i^{(2)}} +e^{4t(s_2+r_i^{(0)})}+e^{4t(x_1+r_i^{(1)})}} + \sum_{i \in C_{22}} \frac{2e^{8t(y_2+ r_i^{(2)})}+e^{4t(s_2+r_i^{(0)})} + e^{4t(x_2+ r_i^{(1)})}}{1+e^{4t(s_2+ r_i^{(0)})}+e^{4t(x_2+r_i^{(1)})} + e^{8t(y_2+ r_i^{(2)})}}, 
	\label{eq:uv1}
	}
	Then 
	\bas{
	\ip{u_3^{(1)}}{\one_n} = \frac{3n}{8} + \sum_{i \in C_{12}}\frac{2e^{4tr_i^{(2)}}+e^{4t(s_1+ r_i^{(0)})} }{1 + e^{4t r_i^{(2)}}+e^{4t(s_1+r_i^{(0)})}} + \sum_{i \in C_{22}} \frac{e^{4t(s_2+r_i^{(0)})} }{1+e^{4t(s_2+ r_i^{(0)})} } + O_P(\sqrt{n}),
	}
	\bas{
	\ip{u_3^{(1)}}{\one_n} \geq \frac{3n}{8} - O_P(\sqrt{n}),
	}
	and 
	\bas{
	\ip{u_3^{(1)}}{\one_n} & \leq \frac{3n}{8} +\sum_{i \in C_{12}}\left( \frac{e^{4t r_i^{(2)}}}{1 + e^{4t r_i^{(2)}}+e^{4t(s_1+r_i^{(0)})}} + \frac{e^{4tr_i^{(2)}}+e^{4t(s_1+ r_i^{(0)})} }{1 + e^{4t r_i^{(2)}}+e^{4t(s_1+r_i^{(0)})}} \right)	\\
	&\qquad \qquad + \sum_{i \in C_{22}} \frac{e^{4t(s_2+r_i^{(0)})} }{1+e^{4t(s_2+ r_i^{(0)})} } + O_P(\sqrt{n})	\\
	& \leq   \frac{3n}{4} + O_P(\sqrt{n})
	}
It follows then
\ba{
 -n/8 - O_P(\sqrt{n}) \leq \ip{u_3^{(1)}}{\one_n}- m \leq n/4+O_P(\sqrt{n}).	
\label{eq:u3v1_bound}
}

	\textbf{Analysis of the second meta iteration:}
	
	We first show that from the previous iteration, the signal $\ip{u_3}{v_2}$ will always dominate $|\ip{u_3}{\one_n} - m|$ which gives desired sign and magnitude of the logits. Then we show the algorithm converges to the true labels after the second meta iteration. 
	
	Using the same decomposition as \eqref{eq:theta10}, 
	\ba{
		s_1^{(2)} =&  (\frac{p+q}{2}-\lambda)(\ip{u_3^{(1)}}{\one_n}- m) + \frac{p-q}{2} \ip{u_3^{(1)}}{v_2}	\\
		& \geq -\frac{n}{8}(\frac{p+q}{2}-\lambda) + \frac{n}{4}\cdot \frac{p-q}{2} - o_P(n \rho_n) \notag \\
		& \geq \frac{n}{8} (\lambda-q) - o_P(n \rho_n) \notag \\
		s_2^{(2)} =& (\frac{p+q}{2}-\lambda)(\ip{u_3^{(1)}}{\one_n}- m) - \frac{p-q}{2} \ip{u_3^{(1)}}{v_2}	\\
		& \leq \frac{n}{4}(\frac{p+q}{2}-\lambda) - \frac{n}{4}\cdot \frac{p-q}{2}+ o_P(n \rho_n) \notag \\
		& = - \frac{n}{4}(\lambda-q) + o_P(n \rho_n),
		\label{eq:logits}
	}
	where we have used Lemma~\ref{prop:pq}. 
	
	After the first meta iteration, the logits satisfy 
	\bas{
		&s_1^{(2)}, -s_2^{(2)} =  \Omega_P(n \rho_n), ~~~~~~x_1^{(1)}, -x_2^{(1)}  = \Omega_P(n \rho_n^{\frac{3}{2}}), \\
		 &y_1^{(1)}, -y_2^{(1)} = \Omega_P(n \rho_n). 
	}
Here we have added the superscripts for the first meta iteration for clarity. 

	In the first iteration of the second meta iteration, $\ip{u_1^{(2)}}{v_{2}}$ is computed as \eqref{eq:uv2} with $s_1$ and $s_2$ replaced with $s_1^{(2)}$ and $s_2^{(2)}$ and the noise replaced accordingly. It is easy to see that 
	\ba{
	\ip{u_1^{(2)}}{v_{2}} \geq \frac{3n}{8} -o_P(n).
	}
	Similarly from \eqref{eq:uv1}, 
		\ba{
		-\frac{n}{8}-o_P(n) \leq \ip{u_1^{(2)}}{\one_n}-m \leq o_P(n).
	}

	The logits are updated as $(\frac{p+q}{2}-\lambda) (\ip{u_1^{(2)}}{\one_n} - m) \pm \frac{p-q}{2}\ip{u_1^{(2)}}{v_2}$, so 
	\ba{
		&x_1^{(2)}, -x_2^{(2)} = \Omega_P(n \rho_n),  
		\label{eq:cond2}
	}
	The same analysis and results hold for $ u^{(2)}_2$ and $(y_1^{(2)}, y_2^{(2)})$.
	We now show after the second meta iteration, in addition to the condition \eqref{eq:cond2}, we further have 
	\ba{
		2y_1^{(2)} - s_1^{(2)} = \Omega_P(n \rho_n),\quad 2y_1^{(2)} - x_1^{(2)} = \Omega_P(n \rho_n)
		\label{eq:cond3}
	}
	
	To simplify notation, let
	\bas{
		\alpha_i(s_1, x_1, y_1) \coloneqq \frac{2e^{8t(y_1+ r_i^{(y)})}+e^{4t(x_1+ r_i^{(x)})}+e^{4t(s_1+ r_i^{(s)})}} {1+e^{4t(s_1+ r_i^{(s)})}+e^{4t(x_1+r_i^{(x)})} + e^{8t(y_1+ r_i^{(y)})}}
	}
where $r$'s are the noise associated with each signal and we have Lemma~\ref{lem:uniform_r} bounding their order uniformly. 

	We first provide an upper bound on $\ip{u_3^{(1)}}{v_2}$. In \eqref{eq:uv2},
	\ba{
		\ip{u_3^{(1)}}{v_2} & \leq \frac{n}{4} +  \sum_{i \in C_{12}}\frac{e^{4t(s_1^{(1)}+ r_i^{(0)})} }{1 +e^{4t(s_1^{(1)}+r_i^{(0)})}}  + \frac{n}{8}  - \sum_{i \in C_{22}} \frac{e^{4t(s_2^{(1)}+r_i^{(0)})} }{1+e^{4t(s_2^{(1)}+ r_i^{(0)})} } + O_P(\sqrt{n})  \notag\\
		& \leq \frac{3n}{8} + \frac{n}{8} \left( \Phi(\frac{-s_2^{(1)} + \Delta}{\sigma_u})  - \Phi(\frac{-s_1^{(1)} - \Delta}{\sigma_u})\right)+ C' n \frac{\rho_u}{\sigma_u^3} + C''ne^{-4t\Delta} + O_P(\sqrt{n}) \notag \\
		& \leq \frac{3n}{8} + o_P(n).
\label{eq:u3v2_ub}	
}
by Lemma \ref{lemma:bounds}.

	For $u^{(2)}_1,$ based on \eqref{eq:uv2} and \eqref{eq:uv1},
	\bas{
		&\ip{u^{(2)}_1}{v_2} = \sum_{i \in C_{11}} \alpha_i(s_1^{(2)}, x_1^{(1)}, y_1^{(1)})  + \frac{n}{4} - o_P(n),  \\
		&\ip{u^{(2)}_1}{\one_n} - m =  \sum_{i \in C_{11}} \alpha_i(s_1^{(2)}, x_1^{(1)}, y_1^{(1)})  -  \frac{n}{4} - o_P(n).
	}
Similarly, 
	\bas{
	&\ip{u^{(2)}_2}{v_2} = \sum_{i \in C_{11}} \alpha_i(s_1^{(2)}, x_1^{(2)}, y_1^{(1)})  + \frac{n}{4} - o_P(n),  \\
	&\ip{u^{(2)}_2}{\one_n} - m =  \sum_{i \in C_{11}} \alpha_i(s_1^{(2)}, x_1^{(2)}, y_1^{(1)})  -  \frac{n}{4} - o_P(n).
}
	
For convenience denote $a=\frac{p+q}{2}-\lambda$ and $b=\frac{p-q}{2}$, then we have	
\bas{
2y_1^{(2)} - s_1^{(2)} & = a(2\ip{u^{(2)}_2}{\one_n}-\ip{u^{(1)}_3}{\one_n}-m) + b(2\ip{u^{(2)}_2}{v_2} - \ip{u^{(1)}_3}{v_2})	\\
& \geq a\left( 2\sum_{i \in C_{11}} \alpha_i(s_1^{(2)}, x_1^{(2)}, y_1^{(1)})  - \frac{n}{4} - m \right) \\
&+ b\left(2 \sum_{i \in C_{11}} \alpha_i(s_1^{(2)}, x_1^{(2)}, y_1^{(1)})  + \frac{n}{2}- \frac{3n}{8}\right) -o_P(n\rho_n)	\\
& = 2(a+b) \sum_{i \in C_{11}} \alpha_i(s_1^{(2)}, x_1^{(2)}, y_1^{(1)})  -  \frac{3an}{8} + \frac{bn}{8} -o_P(n\rho_n)
}
by \eqref{eq:u3v2_ub} and \eqref{eq:u3v1_bound}. Since $\alpha_i(s_1^{(2)}, x_1^{(2)}, y_1^{(1)}) \geq 1 + o_P(1)$, we can conclude
\bas{
2y_1^{(2)} - s_1^{(2)} & \geq \frac{3bn}{8} - \frac{an}{8}- o_P(n\rho_n) = \Omega(n\rho_n).
} 

Similarly, we can check that 
\ba{
2y_1^{(2)} - x_1^{(2)} & = a(2\ip{u^{(2)}_2}{\one_n}-\ip{u^{(2)}_1}{\one_n}-m) + b(2\ip{u^{(2)}_2}{v_2} - \ip{u^{(2)}_1}{v_2})	\notag\\
& = (a+b)\sum_{i \in C_{11}} [2\alpha_i(s_1^{(2)}, x_1^{(2)}, y_1^{(1)}) - \alpha_i(s_1^{(2)}, x_1^{(1)}, y_1^{(1)})] -\frac{(a-b)n}{4} +o_P(n\rho_n)	\notag\\
& \geq  \frac{(b-a)n}{4} - o_P(n\rho_n) = \Omega(n\rho_n)
\label{eq:2y_bound}
}
as $\alpha_i(s_1^{(2)}, x_1^{(2)}, y_1^{(1)})>\alpha_i(s_1^{(2)}, x_1^{(1)}, y_1^{(1)}) $. 
Thus condition \eqref{eq:cond3} holds.

Now we need to analyze the third iteration in this meta iteration. Since $\alpha_i(s_1^{(2)}, x_1^{(2)}, y_1^{(1)}) \leq 2$,
	\bas{
		y_1^{(2)} + y_2^{(2)}  = 2a(\ip{u_2^{(2)}}{\one_n} - m) = o_P(n \rho_n),
	}
then by \eqref{eq:logits}
	\ba{
		s_1^{(2)} - (y_1^{(2)} + y_2^{(2)}) = \Omega_P(n\rho_n), \quad x_1^{(2)} - (y_1^{(2)} + y_2^{(2)})  = \Omega_P(n\rho_n).
		\label{eq:cond4}
	}

Now using the update for $u_3^{(2)}$, and defining the noise in the same way as in the first meta iteration, 
	
	\bas{
		\hspace*{-5cm}
		\ip{u_3^{(2)}}{v_{2}} = & \sum_{i \in C_{11}}\frac{2e^{8t(y_1^{(2)}+ r_i^{(5)})}+e^{4t(x_1^{(2)}+ r_i^{(4)})}+e^{4t(s_1^{(2)}+ r_i^{(3)})}} {1+e^{4t(s_1^{(2)}+ r_i^{(3)})}+e^{4t(x_1^{(2)}+r_i^{(4)})} + e^{8t(y_1^{(2)}+ r_i^{(5)})}}  \\  
		&+\sum_{i \in C_{12}}\frac{e^{4t(s_1^{(2)}+ r_i^{(3)})} - e^{4t(x_2^{(2)}+ r_i^{(4)})}}{1 + e^{4t(y_1^{(2)}+ y_2^{(2)}+ r_i^{(5)})}+e^{4t(s_1^{(2)}+r_i^{(3)})}+e^{4t(x_2^{(2)}+r_i^{(4)})}} \notag \\
		&  + \sum_{i \in C_{21}} \frac{e^{4t(x_1^{(2)}+ r_i^{(4)})}-e^{4t(s_2^{(2)}+r_i^{(3)})}}{1+ e^{4t(y_1^{(2)}+ y_2^{(2)}+ r_i^{(5)})} +e^{4t(s_2^{(2)}+r_i^{(3)})}+e^{4t(x_1^{(2)}+r_i^{(4)})}}  \\
		&- \sum_{i \in C_{22}} \frac{2e^{8t(y_2^{(2)}+ r_i^{(5)})}+e^{4t(s_2^{(2)}+r_i^{(3)})} + e^{4t(x_2^{(2)}+ r_i^{(4)})}}{1+e^{4t(s_2^{(2)}+ r_i^{(3)})}+e^{4t(x_2^{(2)}+r_i^{(4)})} + e^{8t(y_2^{(2)}+ r_i^{(5)})}} \\
		\geq & \sum_{i \in C_{11}}\frac{2e^{8t(y_1^{(2)}+ r_i^{(5)})}} {1+e^{4t(s_1^{(2)}+ r_i^{(3)})}+e^{4t(x_1^{(2)}+r_i^{(4)})} + e^{8t(y_1^{(2)}+ r_i^{(5)})}}  \\
		&+ \sum_{i \in C_{12}}\frac{e^{4t(s_1^{(2)}+ r_i^{(3)})} }{1 + e^{4t(y_1^{(2)}+ y_2^{(2)}+ r_i^{(5)})}+e^{4t(s_1^{(2)}+r_i^{(3)})}+e^{4t(x_2^{(2)}+r_i^{(4)})}} \\
		& +\sum_{i \in C_{21}} \frac{e^{4t(x_1^{(2)}+ r_i^{(4)})}}{1+ e^{4t(y_1^{(2)}+ y_2^{(2)}+ r_i^{(5)})} +e^{4t(s_2^{(2)}+r_i^{(3)})}+e^{4t(x_1^{(2)}+r_i^{(4)})}} \\
		&- n\exp(-\Omega_P(n\rho_n))	\\
		& \geq \frac{n}{2}- n\exp(-\Omega_P(n\rho_n)),
	}
	using the conditions \eqref{eq:cond2}~ \eqref{eq:cond3}~ \eqref{eq:cond4} and Lemma \ref{lem:uniform_r}. 
	Since $\norm{u - z^*}_1 = m - |\ip{u}{v_2}|$, $||u_3^{(2)}- z^*||_1 = n\exp(-\Omega_P(n\rho_n))$ after the second meta iteration.
	
Finally we show the later iterations conserve strong consistency. Since   
	\bas{
		\ip{u_3^{(2)}}{\one}-m = & \sum_{i \in C_{11}}\frac{e^{8t(y_1^{(2)}+ r_i^{(5)})}-1} {1+e^{4t(s_1^{(2)}+ r_i^{(3)})}+e^{4t(x_1^{(2)}+r_i^{(4)})} + e^{8t(y_1^{(2)}+ r_i^{(5)})}}  \\  
		&+\sum_{i \in C_{12}}\frac{e^{4t(y_1^{(2)}+ y_2^{(2)}+ r_i^{(5)})}-1}{1 + e^{4t(y_1^{(2)}+ y_2^{(2)}+ r_i^{(5)})}+e^{4t(s_1^{(2)}+r_i^{(3)})}+e^{4t(x_2^{(2)}+r_i^{(4)})}} \notag \\
		&  + \sum_{i \in C_{21}} \frac{e^{4t(y_1^{(2)}+ y_2^{(2)}+ r_i^{(5)})}-1}{1+ e^{4t(y_1^{(2)}+ y_2^{(2)}+ r_i^{(5)})} +e^{4t(s_2^{(2)}+r_i^{(3)})}+e^{4t(x_1^{(2)}+r_i^{(4)})}}  \\
		&+ \sum_{i \in C_{22}} \frac{e^{8t(y_2^{(2)}+ r_i^{(5)})}-1}{1+e^{4t(s_2^{(2)}+ r_i^{(3)})}+e^{4t(x_2^{(2)}+r_i^{(4)})} + e^{8t(y_2^{(2)}+ r_i^{(5)})}}    \\
		= & n\exp(-\Omega_P(n\rho_n))
	}
by \eqref{eq:cond2}~ \eqref{eq:cond3}~ \eqref{eq:cond4} and Lemma \ref{lem:uniform_r}, we have
\bas{
s_1^{(3)} =&  a(\ip{u_3^{(2)}}{\one}-m) + b \ip{u_3^{(2)}}{v_2} = \frac{p-q}{4}n + n\rho_n \exp(-\Omega_P(n\rho_n)),   \\
s_2^{(3)} =&  a(\ip{u_3^{(2)}}{\one}-m) - b \ip{u_3^{(2)}}{v_2} = -\frac{p-q}{4}n + n\rho_n \exp(-\Omega_P(n\rho_n)).
}
Next we note the noise in this iteration now arises from the whole graph $A$, and can be bounded by
\bas{
r_i^{(7)} & = [R^{zz}, R^{zy}]_{i,\cdot}(u_3^{(2)}-\half \one_n) \\
& = [R^{zz}, R^{zy}]_{i,\cdot}(u_3^{(2)}-z^*) + [R^{zz}, R^{zy}]_{i,\cdot}(z^*-\half\one_n),   \\
}
where the second term is $O_P(\sqrt{n\rho_n \log n})$ uniformly for all $i$, applying Lemma \ref{lem:uniform_r}. To bound the first term, note that 
\bas{
\max_i|[R^{zz}, R^{zy}]_{i,\cdot}(u_3^{(2)}-z^*)| & \leq \|[R^{zz}, R^{zy}](u_3^{(2)}-z^*)\|_2 \\
& \leq O_P(\sqrt{n\rho_n}) \|u_3^{(2)}-z^*\|_1 = o_P(1).
}
Therefore $r_i^{(7)}$ is uniformly $O_P(\sqrt{n\rho_n \log n})$ for all $i$.
By a similar calculation to~\eqref{eq:2y_bound}, we can check that condition ~\eqref{eq:cond3}
holds for $y_1^{(2)}$ and $s_1^{(3)}$, since when $s_1, x_1, y_1 = \Omega(n\rho_n)$ condition ~\eqref{eq:cond3} and $1-o_P(1)\leq \alpha_i(s_1,x_1,y_1) \leq 2+o_P(1)$ guarantees each other and condition ~\eqref{eq:cond3} is true in the previous iteration. We can check that condition \eqref{eq:cond4} also holds. The rest of the argument can be applied to show $\|u_1^{(3)}-z^*\|_1 = n\exp(-\Omega_P(n\rho_n))$. At this point, all the arguments can be repeated for later iterations. 

\end{proof} 

\begin{proof}[\textbf{Proof of Corollary~\ref{cor:not-half}}]
		
We first consider $\mu > 0.5$. By \eqref{eq:s12}, $s_1 = \Omega_P(n \rho_n)$, $s_2 = \Omega_P(n \rho_n)$. Since $r_i^{(0)}=O_P(\sqrt{n\rho_n \log n})$ uniformly for all $i$ by Lemma \ref{lem:uniform_r}, we have 
\begin{align*}
	\phi_i^{(1)} = \frac{1 + e^{4t(s_1 + r_i^{(0)})}}{3+  e^{4t(s_1 + r_i^{(0)})}} = 1-\exp(-\Omega_P(n\rho_n))
\end{align*}
for $i\in C_1$. Similarly for $i\in C_2$,  and $\xi_i^{(1)}=\exp(-\Omega_P(n\rho_n))$. Define $u'_i=\one_{[i\in P_1]}+\one_{[i\in P_2]}$. Since the partition into $P_1$ and $P_2$ is random, $u'_i\sim \text{iid Bernoulli}(1/2)$, and $\|u_1-u'\|_2 = \sqrt{n}\exp(-\Omega_P(n\rho_n))$. 

In the second iteration, we can write 
\begin{align*}
\theta^{01} & = \textstyle 4t([A^{yz}, A^{yy}] - \lambda J)(u_1 - \half\one) \\
& = \textstyle 4t([A^{yz}, A^{yy}] - \lambda J)(u_1-u') + 4t([A^{yz}, A^{yy}] - \lambda J)(u'-\half\one)	\\
& = \textstyle O_P(n\sqrt{\rho}\exp(-\Omega_P(n\rho_n)))+ 4t([A^{yz}, A^{yy}] - \lambda J)(u'-\half\one).
\end{align*}
The signal part of the second term is $4t(x_1\one_{C'_1}+x_2\one_{C'_2})$ with $x_1$ and $x_2$ having the form of~\eqref{eq:s12}, with $u^{(0)}$ replaced by $u'$. Since $x_1, x_2 = \Omega_P(\sqrt{n}\rho_n)$, the rest of the analysis proceeds like that of Theorem~\ref{thm: convergence} restarting from the first iteration.

If $\mu < 0.5$, $s_1 = -\Omega_P(n \rho_n)$, $s_2 = -\Omega_P(n \rho_n)$. We have $\phi_i^{(1)}=\frac{1}{3} + \exp(-\Omega_P(n\rho_n))$, $\xi_i^{(1)}=\frac{2}{3} - \exp(-\Omega_P(n\rho_n))$. This time let $u'=\frac{1}{3}\one_{[i\in P_1]}+\frac{2}{3}\one_{[i\in P_2]}$, then $\theta^{01}$ can be written as
\begin{align*}
\theta^{01} = O_P(n\sqrt{\rho}\exp(-\Omega_P(n\rho_n)))+\frac{4t}{3}([A^{yz}, A^{yy}] - \lambda J)(3u'-\frac{3}{2}\one).
\end{align*}
Noting that $3u'_i-1 \sim \text{iid Bernoulli}(1/2)$, the same argument applies.
\end{proof}

\begin{proof}[\textbf{Proof of Proposition~\ref{prop:stat_pt}}]
(i) We show each point is a stationary point by checking the vector update form of \eqref{eq:thetazo}, \eqref{eq:thetaoz}, \eqref{eq:thetaoo}. Similar to Theorem \ref{thm: convergence}, we have 
\bas{
\theta^{10} = 4t(s_1 \cone + s_2 \ctwo +  r_i^{(0)})
}
where $r_i^{(0)} = O_P(\sqrt{n\rho_n \log n})$. Plugging  $u^{(0)} =  \one_n$ in \eqref{eq:thetaoz}, $s_1 = s_2  = 0.5 (\frac{p+q}{2}-\lambda) n$. Similarly 
\bas{
\theta^{01} =& 4t(x_1 \cone + x_2 \ctwo +  r_i^{(1)}) \\
\theta^{11} =& 4t(y_1 \cone + y_2 \ctwo +  r_i^{(1)}) 
}
where $x_1 = x_2  = 0.5 (\frac{p+q}{2}-\lambda) n$, $y_1 = y_2  =  (\frac{p+q}{2}-\lambda) n$. Plugging in \eqref{eq:phixi} with $\frac{p+q}{2}-\lambda = \Omega_P(\rho_n)$ by Lemma \ref{prop:pq}, we have
\bas{
	\phi_i^{(1)}  = 1-\exp(-\Omega_P(n\rho_n)), \quad \xi_i^{(1)}  = 1-\exp(-\Omega_P(n\rho_n))
}
for all $i \in [m]$. Hence for sufficiently large $n$, $u^{(0)} = \one_n$ is the stationary point. For $u^{(0)} = \mathbf{0}_n$, similarly we have
\bas{
\phi_i^{(1)}  = \exp(-\Omega_P(n\rho_n)), \quad \xi_i^{(1)}  = \exp(-\Omega_P(n\rho_n))
}
so $u^{(0)} = \mathbf{0}_n$ is also a stationary point for large n.

(ii)  The statement for $u^{(0)} = \mathbf{0}_n$ and $u^{(0)} = \one_n$ follows  from Corollary \ref{cor:not-half} by $\mu = 0$ and $\mu = 1$. 

\end{proof}

\begin{proof}[\textbf{Proof of Proposition~\ref{cor:pq_noise}}]
	
	Let $\hat{t}, \hat{\lambda}$ be constants defined in terms of $\hat{p}, \hat{q}$. First we observe using $\hat{p}, \hat{q}$ only replaces $t, \lambda$ with $\hat{t}, \hat{\lambda}$ everywhere in the updates of Algorithm~\ref{alg:M1}. We can check the analysis in Theorem ~\ref{thm: convergence} remains unchanged as long as 
	\begin{enumerate}
		\item $\frac{p+q}{2} > \hat{\lambda}$, 
		\item $\hat{\lambda} - q = \Omega(\rho_n)$,
		\item $\hat{t}=\Omega(1)$.
	\end{enumerate}
	
\end{proof}

\begin{proof}[\textbf{Proof of Theorem~\ref{cor:update_pq}}]

Starting with $p^{(0)}$ and $q^{(0)}$ satisfying the conditions in Corollary \ref{cor:pq_noise}, after two meta iterations of $u$ updates, we have $\|u_3^{(2)}-z^*\|_1 = n\exp(-\Omega(n\rho_n))$. Updating $p^{(1)}, q^{(1)}$ with \eqref{eq:hatpq}, we first analyze the population version of the numerator of $p^{(1)}$, 
\bas{
&(\one_n-u)^T P (\one_n-u) + u^T P u + 2(\one_m-\psi^{10}-\psi^{01})^T \text{diag}(P^{zy})\one_m \\
= & (\one_n-z^*)^T P(\one_n-z^*) + (z^*)^T Pz^* - 2(u-z^*)^TP(\one_n-z^*)+2(z^*)^TP( u-z^*)  \\
& \qquad\qquad
+ (u-z^*)^TP(u-z^*) + O(n\rho_n).    \\
}
In the case of $u_3^{(2)}$, the above becomes
\bas{
\frac{n^2}{2}p + O_P(n^{5/2}\rho_n \exp(-\Omega(n\rho_n))) + O(n\rho_n) = \frac{n^2}{2}p + O_P(n\rho_n). 
}
Next we can rewrite the noise as
\bas{
 & (\one_n-u)^T(A-P)(\one_n-u) + u^T(A-P)u \\
= & (\one_n-z^*)^T (A-P) (\one_n-z^*) + (z^*)^T (A-P)z^* - 2(u-z^*)^T(A-P)(\one_n-z^*) \\
& \qquad\qquad
+2(z^*)^T(A-P)( u-z^*) + (u-z^*)^T(A-P)(u-z^*).  
}
Similarly in the case of $u_3^{(2)}$, the above is $O_P(\sqrt{n^2\rho_n})$. Therefore the numerator of $p^{(1)}$ is $\frac{n^2}{2}p + O_P(\sqrt{n^2\rho_n})$. To lower bound the denominator, note that
	\bas{
		&u^T(J-I)u+(\one-u)^T(J-I)(\one-u)	\\
		= & \left( \sum_{i}u_i \right)^2 + \left(n-\sum_{i}u_i\right)^2-u^Tu -(\one-u)^T(\one-u)	\\
		\geq & n^2/2-2n, 
	}
then we have $p^{(1)} = p+O_P(\sqrt{\rho_n}/n)$. The same analysis shows $q^{(1)} = q+O_P(\sqrt{\rho_n}/n)$.

Replacing $p$ and $q$ with $p^{(1)}$ and $q^{(1)}$ in the final analysis after the second meta iteration of Theorem \ref{thm: convergence} does not change the order of the convergence, and the rest of the arguments can be repeated.
\end{proof}

\section{Generalizations}
\label{sec:gen}
We present the \textbf{update equations for balanced $K>2$} models. We will use the notation $a,b\in\{0,\dots, K-1\}$ to be consistent with the two class case.
Let $S_{zy}=2t(\text{diag}(A^{zy})-\lambda I) \mathbf{1}_m$.
\ba{\label{eq:thetaK}
\theta^{ab} =\begin{cases}  \textstyle 2t[A^{zz} - \lambda(J-I)](\phi_a - \phi_0) + 2t[A^{zy}-\lambda(J-I)-\text{diag}(A^{zy})](\xi_a - \xi_0)  - S_{zy},    & \mbox{$a\neq 0, b=0$}\\ 
 \textstyle 2t[A^{zz} - \lambda(J-I)](\phi_b - \phi_0) + 2t[A^{zy}-\lambda(J-I)-\text{diag}(A^{zy})](\xi_b - \xi_0)  - S_{zy},    & \mbox{$a= 0, b\neq 0$}\\ 
\theta_{a0}+\theta_{b0}+S_{zy} & \mbox{$a\neq 0, b\neq 0$ }
\end{cases}
}
The  \textbf{update equations for unbalanced two class blockmodels} simply adds an additional term of $\log \pi/(1-\pi)$ to the updates of $\theta_{10}$ (Eq.~\eqref{eq:thetaoz}), $\theta_{01}$ (Eq.~\eqref{eq:thetazo}) and $2\log \pi/(1-\pi)$ to $\theta_{11}$ (Eq.~\eqref{eq:thetaoo}).
We assume that the proportions are known.

In Figure~\ref{fig:heatmap_a}, we show the heatmap for mis-specified parameters for VIPS on unbalanced SBM $(\pi=.3)$ and balanced SBM with $K=3$. For each starting point of $\hat{p},\hat{q}$ the average NMI is shown. We see that in both cases the VIPS algorithm converges to the correct labels for a wide range of initial parameter settings.  
\begin{figure}[H]
\centering
\begin{tabular}{cc}
\hspace{-0.5cm}
(a)~Unbalanced with $\pi=0.3$& (b)~Balanced with $K=3$ \\
 \includegraphics[width=0.42\textwidth]{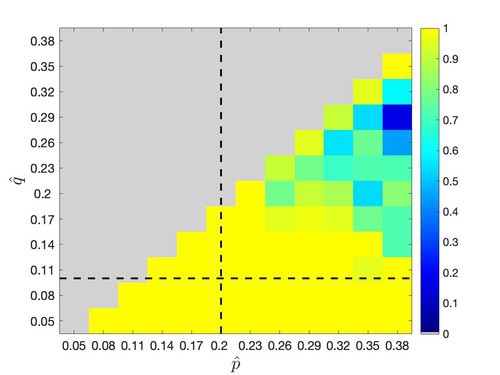}& 
\includegraphics[width=0.42\textwidth]{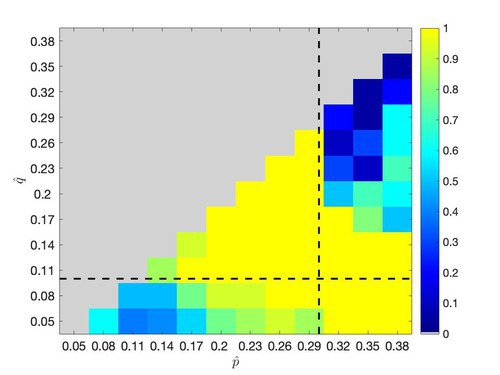}
\end{tabular}
\caption{
NMI with different estimation of $\hat{p}$, $\hat{q}$ with $\hat{p} > \hat{q}$, averaged over 20 random initializations for each $\hat{p}$, $\hat{q}$. The left figure has $\pi=0.3, K=2$ and the right figure has balanced clusters with $K=3$. The true $(p_0, q_0) = (0.2, 0.1)$ and $n=2000$. 
}
\vspace{-3mm}
 \label{fig:heatmap_a}
\end{figure}

For $K=3$, we also show Figure \ref{fig:k3_a}, where each row represents the estimated membership of one random trial and both MFVI and VIPS are run with the true $p_0, q_0$. We show VIPS can recover true membership with higher probability than MFVI.

\begin{figure}[H]
    \centering
\begin{tabular}{cc}
 \hspace{-0.2cm}(a)~MFVI& (b)~VIPS \\ 
  \hspace{-0.2cm}\includegraphics[width=0.42\textwidth]{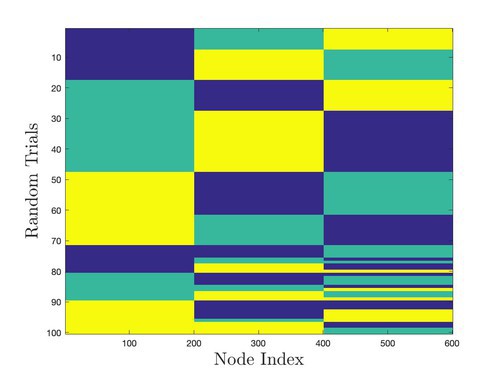} & 
\includegraphics[width=0.42\textwidth]{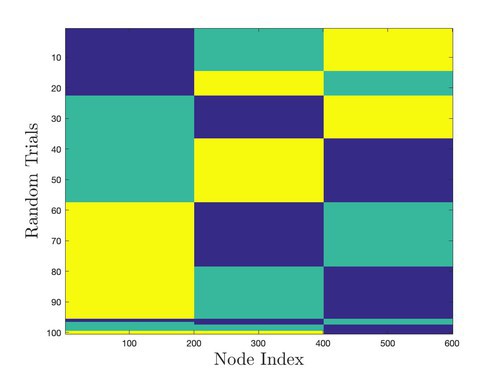} 
\end{tabular}
  \caption {Compare VIPS and MFVI when K=3, equal sized communities, for known $p_0, q_0$ in 100 random trials. $p_0 = 0.5, q_0 = 0.01$. Rows permuted for visual clarity. }
  \label{fig:k3_a}
\end{figure}

\section{Additional Experimental Results}
\label{sec:expaddl}
In Figure \ref{fig:converge1}, we compare different update rules in VIPS with (i) parameters $p,q$ fixed at the true values (True), (ii) $(p,q)$ estimated using $(\sum_{i\neq j}A_{ij}/(n(n-1)),\sum_{i\neq j}A_{ij}/(2n(n-1)) )$ but fixed (Estimate), and (iii) $(p,q)$ initialized as in (ii) and updated in the algorithm (Update) using Eq.~\eqref{eq:hatpq}. In all settings, VIPS successfully converges to the ground truth, which is consistent with our theoretical results and shows robustness of the parameter setting. 
\begin{figure}[H]
\centering
\begin{tabular}{ccc}
\hspace{-0.5cm}
 \includegraphics[width=0.32\textwidth]{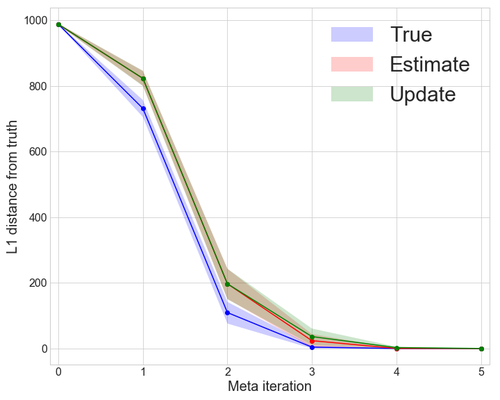}& 
 \hspace{-0.5cm}
 \includegraphics[width=0.32\textwidth]{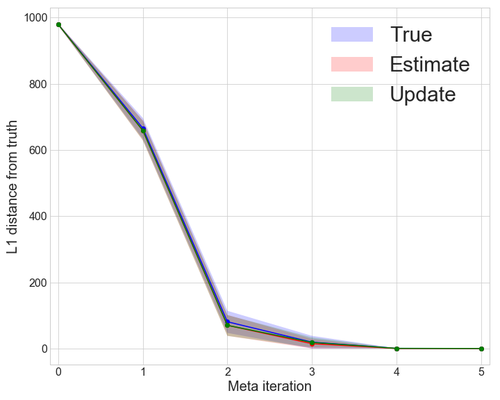}&
 \hspace{-0.5cm}
  \includegraphics[width=0.32\textwidth]{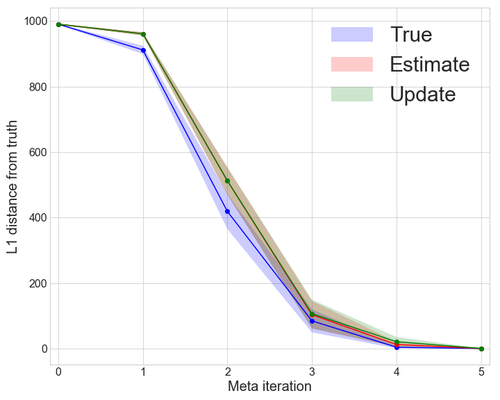}
 \\
\end{tabular}
 \caption{
Values of  $\norm{u - z^*}_1$ as the number of meta iterations increases. Each line is the mean curve of 50 random trials and the shaded area is the standard deviation. Here $n =2000$ and $p_0= 0.1, q_0 = 0.02$. $u$ is initialized by Bernoulli distribution with mean $\mu = 0.1, 0.5, 0.9$ from the left to right. 
}
\label{fig:converge1}
 \end{figure}

\iffalse
In Figure \ref{fig:heatmap_a}, we demonstrate the robustness of VIPS to parameter estimation in the unbalanced community with $\pi = 0.3$. 

\begin{figure}[H]
\centering
\begin{tabular}{cc}
\hspace{-0.5cm}
(a)~MFVI& (b)~VIPS \\
 \includegraphics[width=0.43\textwidth]{newnew/fig4.jpg}& 
\includegraphics[width=0.43\textwidth]{newnew/fig6.jpg}
\end{tabular}
\caption{
NMI with different estimation of $\hat{p}$, $\hat{q}$ with $\hat{p} > \hat{q}$, averaged over 50 random initializations for each $\hat{p}$, $\hat{q}$. The true $(p_0, q_0) = (0.2, 0.1)$, $\pi = 0.3$ and $n=2000$. 
}
 \label{fig:heatmap_a}
\end{figure}
\fi

In Figure~\ref{fig:compareupdate}, we compare VIPS and MFVI with and without parameter updates. For VIPS, we do parameter updates from 3rd meta iteration onward, and for fairness, we start parameter updates  9 iterations onward for MFVI. In both schemes, the VIPS performs better than MFVI.
\begin{figure}[H]
\centering
\begin{tabular}{cc}
\hspace{-0.5cm}
 \includegraphics[width=0.43\textwidth]{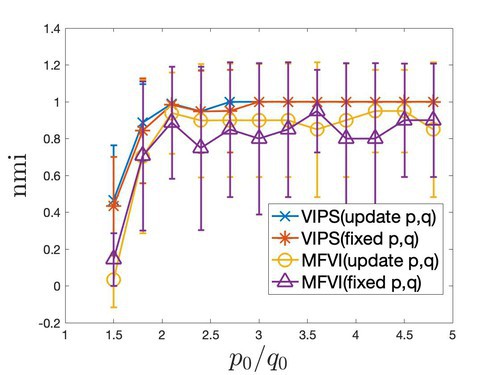}& 
\includegraphics[width=0.43\textwidth]{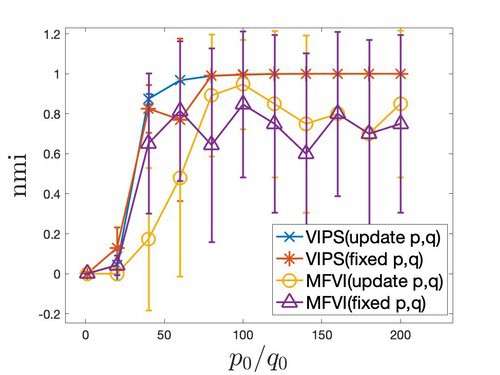}
\end{tabular}
\caption{
\label{fig:compareupdate}Two schemes for estimating model parameters for VIPS and MFVI. Both use the initial $\hat{p}$ and $\hat{q}$ as described in Figure 4 in the main paper. The first scheme starts updating $\hat{p}$ and $\hat{q}$ after 3 meta iterations for VIPS and 9 iterations for MFVI. The other scheme has $\hat{p},\hat{q}$ held fixed. 
}
\end{figure}

\section*{Acknowledgement}
P.S. was partially funded by NSF grant DMS 1713082.
\bibliography{reference.bib}
\bibliographystyle{unsrt}

\end{document}